\documentclass[11pt,a4]{article}

\usepackage{amssymb}
\usepackage{amsmath}

\usepackage[latin1]{inputenc}

\usepackage{graphicx}

\newcommand\bigforall{\mbox{\LARGE $\mathsurround0pt\forall$}}

\usepackage{tikz}
\usetikzlibrary{arrows,shapes.gates.logic.US,shapes.gates.logic.IEC,calc}

\usepackage{xy}
\xyoption{import}

\usepackage{hyperref}

\DeclareFontFamily{U}{mathx}{\hyphenchar\font45}
\DeclareFontShape{U}{mathx}{m}{n}{
      <5> <6> <7> <8> <9> <10>
      <10.95> <12> <14.4> <17.28> <20.74> <24.88>
      mathx10
      }{}
\DeclareSymbolFont{mathx}{U}{mathx}{m}{n}
\DeclareFontSubstitution{U}{mathx}{m}{n}
\DeclareMathAccent{\widecheck}{0}{mathx}{"71}
\DeclareMathAccent{\wideparen}{0}{mathx}{"75}

\newcommand{\PPL}                		{{\mathsf{PPL}}}
\newcommand{\CPL}                		{{\mathsf{CPL}}}
\newcommand{\svPL}                	 {{\mathsf{svPL}}}
\newcommand{\RCOF}              	 {{\mathsf{RCOF}}}
\newcommand{\inte}[1]               	  {{\mathop{\textstyle\int\!{#1}}}}

\newcommand{\pwp}[1]                 	{\mathop{\wp^+{#1}}}
\newcommand{\fwp}[1]                 	{\mathop{\wp_\mathsf{fin}{#1}}}
\newcommand{\pfwp}[1]                 	{\mathop{\wp^+_\mathsf{fin}{#1}}}
\newcommand{\satp}[1]                 	{\mathbin{\sat_{#1}}}
\newcommand{\entpq}[2]                 	{\mathbin{\ent_{#1}^{#2}}}
\newcommand{\hentpq}[2]                 {\mathbin{\ddot{\ent}_{#1}^{#2}}}
\newcommand{\mycirc}               	{\mathbin{@}}

\newcommand{\dummy}                {{\text{ }}}
\newcommand{\ds}                   {\displaystyle}

\newcommand{\qurtains}             {\hfill{QED}}

\ifx \theorem \undefined
\newtheorem{definition}{\vspace{1mm}Definition}[section]

\newtheorem{pth}[definition]{\vspace{1mm}Theorem}

\newtheorem{prop}[definition]{\vspace{1mm}Proposition}

\newenvironment{proof}{\begin{trivlist}\item{\bf Proof:}}{\qurtains\end{trivlist}}

\newcommand{\SVP}[1]					 {{\widecheck{#1}}}
\newcommand{\PSV}[1]					 {{\widehat{#1}}}

\newcommand{\Prob}					{\textsf{Prob}}

\newcommand{\lrule}[2] 				{\frac{\ \ #1\ \ }{\ \ #2\ \ }}

\newcommand{\pnats}                 {{\mathbb{N}^+}}
\newcommand{\nats}                  {{\mathbb N}}

\newcommand{\reals}                 {{\mathbb R}}

\newcommand{\termvalue}[1]          {{\lbrack\!\lbrack #1 \rbrack\!\rbrack}}
\newcommand{\tv}[1]                 {{\termvalue{#1}}}

\newcommand{\inv}[1]                {{{#1}^{-1}}}

\newcommand{\from}                  {{{\leftarrow}}}

\newcommand{\lverum}                {{\mathsf{t\!t}}}
\newcommand{\lfalsum}               {{\mathsf{f\!f}}}

\newcommand{\lneg}                  {\mathop{\neg}}

\newcommand{\limp}                  {\mathbin{\supset}}
\newcommand{\leqv}                  {\mathbin{\equiv}}
\newcommand{\lconj}                 {\mathbin{\wedge}}
\newcommand{\ldisj}                 {\mathbin{\vee}}

\newcommand {\num}[1]              {{#1}}

\newcommand{\satc}                    {\sat_{\mathsf{c}}}
\newcommand{\entc}                    {\ent_{\mathsf{c}}}
\newcommand{\sat}                    {\Vdash}
\newcommand{\ent}                    {\vDash}
\newcommand{\der}                    {\vdash}

\newcommand{\satfo}                  {\mathbin{\sat_{\mathsf{fo}}}}

\newcommand {\HYP}		{\text{HYP}}
\newcommand {\TAUT}                 {\text{TT}}

\newcommand {\MP}                   {\text{MP}}

\newcommand {\cA}                   {\mathcal{A}}

\newcommand {\RR}                   {\text{RR}}
\newcommand {\DNF}                   {\text{DNF}}

\newcommand{\mea}           {\mu}
\newcommand{\fdd}           {\eta}
\newcommand{\cF}           {\mathcal{F}}

\begin{document}

\title{On probability and logic}

\author{A.~Sernadas\ \ J.~Rasga\ \ C.~Sernadas\\[1mm]
{\scriptsize Departamento de Matemática, Instituto Superior Técnico}\\[-1mm] 
{\scriptsize Centro de Matemática, Aplicações Fundamentais e Investigação Operacional}\\ [-1mm]
{\scriptsize Universidade de Lisboa, Portugal}\\
{\tiny \{acs,jfr,css\}@math.tecnico.ulisboa.pt}}

\date{December 5, 2015}
 
\maketitle

\begin{abstract} 
Within classical propositional logic, assigning probabilities to formulas is shown to be equivalent to assigning probabilities to valuations.
A novel notion of probabilistic entailment enjoying desirable properties of logical consequence is proposed and shown to collapse into the classical entailment when the language is left unchanged. 
Motivated by this result, a decidable conservative enrichment of propositional logic is proposed by giving the appropriate semantics to a new language construct that allows the constraining of the probability of a formula. 
A sound and weakly complete axiomatization is provided using the decidability of the theory of real closed ordered fields.
\\[2mm]
{\bf Keywords}: probabilistic propositional logic, stochastic valuation, probabilistic entailment, decidability.\\[2mm]
{\bf AMS MSC2010}: 03B48, 03B60.
\end{abstract}


\section{Introduction}\label{sec:introduction}

Starting as far back as~\cite{ram:26} and~\cite{car:50},
adding probability features to logic has been a recurrent research topic. 

The introduction of probabilities in formal logic is quite challenging since there is the need to accommodate the continuous nature of probabilities within the discrete setting of symbolic reasoning. 
It is also interesting from the practical point of view since probabilistic reasoning is very relevant in many fields such as philosophy, economics, computer science and artificial intelligence. 

Several ways to combine probabilities and  logic have been considered. 
One can assign probabilities either to formulas or to models. 
One can either keep the original language unchanged by introducing probabilities only at the meta-level or change the language in order to internalize probabilistic assertions. 

For seminal examples of assigning probabilities to formulas while leaving the formal language unchanged see
\cite{nil:86,hai:96,adam:98,hai:11}.
Under this approach, a probabilistic entailment is defined by relating the probabilities of the hypotheses and the probability of the conclusion. 

The approach of assigning probabilities to models was first explored in~\cite{bur:69} and subsequently revisited by several authors, all of them choosing to change the original language in order to be able to express probabilistic assertions. Two techniques were considered when endowing models with probabilities. 

The ``endogenous'' technique adopted by  most authors consists of enriching each model of the original logic with a probability measure on some components. For instance, in modal-like logic this approach was followed for Kripke structures by assigning probabilities to worlds ~\cite{bur:69,fag:hal:94} or to the pairs in the accessibility relation~\cite{hoe:92,ben:09}. This technique has been quite pervasive in probabilistic versions of logics for reasoning about computer programs involving random operations, for example in~\cite{har:84,har:02,rchadha:lcf:pmat:acs:06}. It was also used in~\cite{kei:85,aba:94,ras:00,lcf:jfr:acs:css:06,kei:09,jfr:wl:css:13} for probabilizing predicate logic by assigning probabilities to the individuals in the domain.

The ``exogenous'' technique  consists of assigning a probability to each model of the original logic or to a class of models of the original logic~\cite{nil:86,bac:90,hal:90,aba:94,hei:01,hal:05,pmat:acs:css:05a,roc:08,spe:13}. A similar technique was used in~\cite{pmat:acs:05} for assigning amplitudes to models in order to set-up a logic for reasoning about quantum systems.

The existence of so diverse proposals of incorporating probability into formal logic raises the problem of expressivity namely, for instance, if nesting probability operators will be more expressive. The negative answer is given in~\cite{fin:14}.

In this paper, within the setting of propositional logic, 
we propose in Section~\ref{sec:svs} the novel notion of stochastic valuation and its main properties. 
The section ends with the proof of the equivalence of assigning probabilities to formulas and assigning probabilities to valuations.
Afterwards, we address two problems using the latter approach via stochastic valuations. 

First, leaving the language unchanged, in Section~\ref{sec:pent} we start by criticizing the definition of probabilistic entailment introduced in~\cite{hai:96,hai:11} and then propose a novel notion enjoying the usual properties of a logical consequence. The section ends with the proof of the collapse of probabilistic entailment into classical entailment.
 
Second, since nothing is gained by probabilizing formulas or models while keeping the language unchanged, 
we propose in Section~\ref{sec:PPL} a small enrichment ($\PPL$) of propositional logic by providing the appropriate stochastic-valuation semantics to a new language constructor that allows (without nesting) the constraining of the probability of a formula. 
At the end of Section~\ref{sec:PPL}, capitalizing on the decidability of the theory of real closed ordered fields, we present an axiomatization of $\PPL$.
This axiomatization is shown to be sound and weakly complete in Section~\ref{sec:soundcomp}.
Moreover, we prove in Section~\ref{sec:cons} that $\PPL$ is a decidable conservative extension of classical propositional logic. 

The paper ends with an assessment of what was achieved and a brief discussion of possible future work in Section~\ref{sec:concs}.

\section{Stochastic valuations}\label{sec:svs}

Towards endowing propositional logic with a probabilistic semantics, we introduce here the notion of stochastic valuation and show that it induces a probability assignment to formulas that fulfils the principles postulated 
in~\cite{adam:98}. We also show that each probability assignment to formulas fulfilling those principles induces a unique stochastic valuation that recovers the original assignment. These results allow us to conclude that the choice of assigning probabilities to valuations or to formulas is immaterial. In the subsequent sections of this paper we stick to the approach of assigning probabilities to valuations using stochastic valuations.

Throughout the paper, $L$ is the propositional language generated by the set $B=\{B_j: j \in \nats\}$ of propositional symbols  using the connectives $\lneg$ and $\limp$. 
The other connectives, as well as $\lverum$ (verum) and $\lfalsum$ (falsum), are introduced as abbreviations as usual. Recall that a (classical) valuation is a map $v: B \to \{0,1\}$. Given $A \in \pfwp B$, we say that an $A$-valuation is a map $u: A \to \{0,1\}$.\footnote{Given a set $S$, we denote 
the collection of its subsets by $\wp S$,
the collection of its non-empty subsets by $\pwp S$,
the collection of its finite subsets by $\fwp S$
and
the collection of its non-empty finite subsets by $\pfwp S$.}
We use $v \satc \alpha$ for stating that valuation $v$ satisfies formula $\alpha$ and $\Delta \entc \alpha$ for stating that
$\Delta$ entails $\alpha$, that is $v \satc \alpha$ whenever  $v \satc \delta$ for every $\delta \in \Delta$.

When probabilizing valuations one might be tempted to look at probabilistic valuations as random variables taking values on the set of all classical valuations. However, it turns out that it is much better to look at them as stochastic processes as follows.

By a {\it stochastic valuation} $V$ we mean a collection $\{V_{B_j}: B_j \in B\}$ of discrete  
random variables defined on a probability space $(\Omega, \cF,\mea)$ and
taking values in $\{0,1\}$.

In other words, such a $V$ is a stochastic process indexed by $B$ where
$\Omega$ is the set of outcomes,
$\cF \subseteq \wp\Omega$ is the $\sigma$-field of events,
$\mea:\cF \to [0,1]$ is the probability measure
and each $V_{B_j}:\Omega\to\{0,1\}$ is a measurable map (that is, such that $\inv{(V_{B_j})}(S) \in \cF$ for every $S\subseteq\{0,1\}$). For further details on the notion of stochastic process consult, for instance,~\cite{bil:12}.

For the purposes of this paper it is convenient to identify each valuation $v$ 
with the subset $\{B_j:v(B_j)=1\}$ of $B$. Accordingly, restriction is achieved by intersection:
given a subset $A$ of $B$,
$v|_A = v \cap A$.

Moreover, it becomes handy to assume that 
each random variable $V_{B_j}$ takes values in $\{\emptyset,\{B_j\}\}$ 
with
$\emptyset$ standing for $0$
and
$\{B_j\}$ for $1$.

Then, given a non-empty finite subset $A=\{B_{j_1},\dots,B_{j_n}\}$ of $B$ and $U \subseteq A$,
we write
$$\Prob(V_A =U)$$
for the (joint) probability (given by $V$)
$$\mea\left(\bigcap_{k=1}^n\inv{(V_{B_{j_k}})}(U \cap \{B_{j_k}\})\right)$$
of each $B_{j_k} \in U$ being true and each $B_{j_k} \in A \setminus U$ being false.

In particular,
$$\Prob(V_{B_j} = \{B_j\}) = \mea\left(\inv{(V_{B_j})}(\{B_j\})\right)$$
is the probability (given by $V$) of $B_j$ being true
while
$$\Prob(V_{B_j} = \emptyset) = \mea\left(\inv{(V_{B_j})}(\emptyset)\right)$$
is the probability (given by $V$) of $B_j$ being false. 

Each stochastic valuation $V$ induces the family
$$\{U \mapsto \Prob(V_A =U): \wp A \to [0,1]\}_{A\in\pfwp B}$$
of finite-dimensional (joint probability) distributions that we may call {\it finite-dimensional probabilistic valuations}.

This family is consistent in the sense that the following {\it marginal condition} holds:
$$\Prob(V_{A'}=U') =\sum_{\text{\shortstack[c]
						{$U \subseteq A$\\
						$U {\cap} A'{=}U'$
					  }}}  						\Prob(V_A=U)
	\;\; \forall A {\in} \pfwp B\; \forall A' {\in} \pwp A\; \forall U' {\in} \wp A'.$$

Conversely, given a consistent system of finite-dimensional distributions (in our case, finite-dimensional probabilistic valuations), 
the Kolmogorov existence theorem (see Section 36 of~\cite{bil:12}) guarantees the existence of a 
unique stochastic process (in our case, a unique stochastic valuation) 
that induces those finite-dimensional distributions. This theorem will be frequently used in the paper.
Its availability well justifies our claim that it is much better to probabilize valuations using stochastic processes.

Given $\alpha \in L$, let
$B_\alpha$ be the set of propositional symbols occurring in $\alpha$ and 
$\termvalue{\alpha}$ be the set $\{v \cap B_\alpha: v \satc \alpha\}$ of the restrictions to $B_\alpha$ of the valuations that satisfy $\alpha$. 

With these notions and notation at hand we are ready to compute the probability that a stochastic valuation assigns to a formula.

Given $\alpha \in L$ and a stochastic valuation $V$, 
the {\it probability of} $\alpha$ {\it under} $V$, is computed as follows:
$$\Prob_{V}(\alpha)=\sum_{U \in \termvalue{\alpha}} \Prob(V_{B_\alpha}=U).$$ 
That is, the probability of $\alpha$ under $V$ is the sum of the probabilities of the restrictions to $B_\alpha$ of the classical valuations that satisfy $\alpha$. These probabilities are provided by the finite-dimensional probabilistic valuation 
$$U \mapsto \Prob(V_{B_\alpha} =U): \wp B_\alpha \to [0,1]$$
induced by $V$ on $B_\alpha$.


The probabilities that a stochastic valuation assigns to formulas fulfil the principles postulated by Adams 
(in~\cite{adam:98}) as we now proceed to show.
Recall that according to Adams, a {\it probability assignment} is a map $P$ on $L$ satisfying the following principles:
\begin{itemize}
\item [P1] $0 \leq P(\alpha) \leq 1$;
\item [P2] If $\entc \alpha$ then $P(\alpha)=1$;
\item [P3] If $\alpha \entc \beta$ then $P(\alpha) \leq P(\beta)$;
\item [P4] If $\entc \lneg (\beta \lconj \alpha)$ then $P(\beta \ldisj \alpha)= P(\beta) + P(\alpha)$.
\end{itemize}

Some notation and a few auxiliary results are needed before showing that the probability assignment to formulas induced by a stochastic valuation does indeed fulfil these four principles.

Given $U \subseteq A \subseteq B$, we use the abbreviation
$$\phi_A^U \quad\text{for}\quad
		\left(\bigwedge_{B_j \in U} B_j\right) \lconj \left(\bigwedge_{B_j \in A \setminus U} \lneg B_j\right).$$
Clearly, this formula identifies the $A$-valuation that makes each $B_j$ in $U$ true and each $B_j$ not in $U$ false.
Observe that, for each such $U$ and $A$, the set
$$\{v\cap A: v \satc \phi_A^U\}$$
is the singleton $\{U\}$. 
Remark also that the set $B_{\phi_A^U}$ of propositional symbols occurring in $\phi_A^U$ coincides with $A$.

\begin{prop}\em \label{prop:contUdif}
Let $U_1,U_2 \subseteq A \subseteq B$ be such that $U_1 \neq U_2$. Then
$$\entc \lneg (\phi_A^{U_1} \lconj \phi_A^{U_2})$$
\end{prop}
\begin{proof}\ \\
Let $v$ be a valuation. Assume, by contradiction,  that $v \satc \phi_A^{U_1} $ and $v \satc \phi_A^{U_2}$. 
Without loss of generality, let $B_j$ be a  symbol in $U_1$ but not in $U_2$. Then, $v \satc B_j$ and $v \satc \lneg B_j$
which is a contradiction.
\end{proof}

\begin{prop}\em \label{prop:disjtottaut}
Let $A \subseteq B$. Then
$$\entc \bigvee_{U \subseteq A} \phi_A^U.$$
\end{prop}
\begin{proof}\ \\
Let $v$ be a valuation. Then, it is straightforward to see that $v \satc \phi_A^{v \cap A}$.
\end{proof}

\begin{prop}\em \label{prop:marglog}
Let $U' \subseteq A' \subseteq A \subseteq B$. Then
$$\entc \left(\bigvee_{\text{\shortstack[c]
{$U \subseteq A$\\
$U \cap A'=U'$
}
}}  \phi_A^U\right) \leqv \phi_{A'}^{U'}.$$
\end{prop}
\begin{proof}\ \\
Let $v$ be a valuation. \\
$(\to)$ 
Assume that $$v \satc \left(\bigvee_{\text{\shortstack[c]
{$U \subseteq A$\\
$U \cap A'=U'$
}
}}  \phi_A^U\right).$$
Then, $v \satc \phi_A^U$ for some $U \subseteq A$ such that $U \cap A'=U'$. Let $B_j \in U'$. Then, $B_j \in U \cap A'$ and so
$v \satc B_j$ since $v \satc \phi_A^U$. Let $B_j \in A' \setminus U'$. Then, $B_j \not \in U \cap A'$ and so $B_j \not \in U$. Hence, 
$v \not \satc B_j$ since $v \satc \phi_A^U$. Thus, $v \satc \phi_{A'}^{U'}$.\\[2mm]
$(\from)$ Assume that $v \satc \phi_{A'}^{U'}$. Observe that $v \satc \phi_A^{v \cap A}$. Moreover, $v \cap A \subseteq A$ and
$(v \cap A) \cap A' = v \cap A' =U'$ since $v \satc \phi_{A'}^{U'}$. 
\end{proof}
 
\begin{prop} \em\label{prop:mardisj}
Let $\delta,\phi$ be formulas and $V$ a stochastic valuation. Then
$$\displaystyle \sum_{U \in \{v \cap (B_\delta \cup B_\phi): v \satc \delta\}} \Prob(V_{B_\delta \cup B_\phi}=U)= \sum_{U'\in \{v \cap B_\delta: v \satc \delta\}} \Prob( V_{B_\delta}=U').$$
\end{prop}
\begin{proof}\ \\
Observe that:
\begin{align*} \displaystyle \sum_{U'\in \{v \cap B_\delta: v \satc \delta\}} \Prob( V_{B_\delta}=U') 
\displaybreak[1] \\[1mm]
& \hspace*{-40mm} \displaystyle = \sum_{U'\in \{v \cap B_\delta: v \satc \delta\}}\sum_{\text{\shortstack[c]
{$U \subseteq B_\delta \cup B_\phi$\\
$U \cap B_\delta=U'$
}
}} 
\Prob(V_{B_\delta \cup B_\phi}=U) & (*) \displaybreak[1] \\[1mm]
& \hspace*{-40mm} \displaystyle = \sum_{\text{\shortstack[c]
{$U \subseteq B_\delta \cup B_\phi$\\
$U \cap B_\delta \in \{v \cap B_\delta: v \satc \delta\}$
}
}} 
\Prob(V_{B_\delta \cup B_\phi}=U)  \displaybreak[1] \\[1mm]
& \hspace*{-40mm} \displaystyle =\sum_{U \in \{v \cap (B_\delta \cup B_\phi): v \satc \delta\}} \Prob(V_{B_\delta \cup B_\phi}=U) & (**) 
\end{align*}
where $(*)$ follows by the marginal condition and $(**)$ is proved now. Indeed
$$U \subseteq B_\delta \cup B_\phi \text{ and } U \cap B_\delta \in \{v \cap B_\delta: v \satc \delta\} \; \text{ iff } \;
U \in \{v \cap (B_\delta \cup B_\phi): v \satc \delta\}$$ since:\\[1mm]
$(\to)$ Assume that $U \subseteq B_\delta \cup B_\phi \text{ and } U \cap B_\delta \in \{v \cap B_\delta: v \satc \delta\}$. Let $v$ be such that $U \cap B_\delta= v \cap B_\delta$ and $v \satc \delta$. Let $v'$ be such that $v' \cap B_\delta = v \cap B_\delta$ and 
$v' \cap (B_\delta \cup B_\phi)= U$. Then, $v' \satc \delta$. Therefore, $U \in \{v \cap (B_\delta \cup B_\phi): v \satc \delta\}$.
\\[2mm]
$(\from)$ Assume that 
$U \in \{v \cap (B_\delta \cup B_\phi): v \satc \delta\}$. Let $v$ be such that $U= v \cap (B_\delta \cup B_\phi)$ and $ v\satc \delta$. Thus,
$U \subseteq B_\delta \cup B_\phi$. 
Moreover 
$U \cap B_\delta= v \cap (B_\delta \cup B_\phi) \cap B_\delta=v \cap B_\delta$. Hence
$U \cap B_\delta \in \{v \cap B_\delta: v \satc \delta\}$.
\end{proof}

\begin{prop} \em\label{prop:monprobs}
Let $\delta, \alpha$ be formulas such that $\delta \entc \alpha$. Then
$$\Prob_{V}(\delta) \leq \Prob_{V}(\alpha).$$
\end{prop}
\begin{proof}\ \\
Observe that
\begin{align*}
\Prob_{V}(\delta) = &\;   \displaystyle \sum_{U' \in \termvalue{\delta}} \Prob(V_{B_\delta} =U') \displaybreak[1] \\[2mm]
                                 = & \;  \displaystyle \sum_{U' \in \termvalue{\delta}}  \quad \sum_{\text{\shortstack[c]
                                           {$U \subseteq  B_\delta \cup B_\alpha$\\
                                               $U \cap B_\delta = U'$
}
}} 
 \Prob(V_{B_\delta \cup B_\alpha} =U) & (*)\displaybreak[1] \\[2mm]
                                = & \;   \displaystyle         \sum_{\text{\shortstack[c]
                                           {$U \subseteq  B_\delta \cup B_\alpha$\\
                                               $U \cap B_\delta\in \termvalue{\delta}$
}
}}  \Prob(V_{B_\delta \cup B_\alpha} =U)\displaybreak[1] \\[2mm]                                           
                                \leq &   \;                \displaystyle         \sum_{\text{\shortstack[c]
                                           {$U \subseteq  B_\delta \cup B_\alpha$\\
                                               $U \cap B_\alpha\in \termvalue{\alpha}$
}
}}  \Prob(V_{B_\delta \cup B_\alpha} =U) & (**) \displaybreak[1]  \\[2mm]
                                  = &  \;  \displaystyle \sum_{U' \in \termvalue{\alpha}}  \quad \sum_{\text{\shortstack[c]
                                           {$U \subseteq  B_\delta \cup B_\alpha$\\
                                               $U \cap B_\alpha = U'$
}
}} 
 \Prob(V_{B_\delta \cup B_\alpha} =U)\displaybreak[1]  \\[2mm]
                                = & \;   \displaystyle \sum_{U' \in \termvalue{\alpha}} \Prob(V_{B_\alpha} =U') &  (*)\displaybreak[1] \\[2mm]
                                = & \; \Prob_{V}(\alpha)
\end{align*}
where $(*)$ follow from the marginal condition on $V$ and $(**)$ comes from the fact that
$$\{U: U \subseteq B_\delta \cup B_\alpha, U \cap B_\delta \in \termvalue{\delta}\} 
\subseteq \{U: U \subseteq B_\delta \cup B_\alpha, U \cap B_\alpha \in \termvalue{\alpha}\}$$
as we now show.
Let $U \subseteq B_\delta \cup B_\alpha$ be such that $U \cap B_\delta \in \termvalue{\delta}$. Let $v$ be a valuation such that
$$v \cap (B_\delta \cup B_\alpha)=U.$$
Then, $v \cap B_\delta= v \cap (B_\delta \cup B_\alpha) \cap B_\delta= U\cap B_\delta$ and so 
$v \cap B_\delta \in  \termvalue{\delta}$. Hence, $v \satc \delta$ and thus $$v \satc \alpha$$ since $\delta \entc \alpha$.
Therefore, $v \cap B_\alpha \in \termvalue{\alpha}$. Since
$$v \cap B_\alpha= v \cap (B_\delta \cup B_\alpha) \cap B_\alpha= U\cap B_\alpha$$
then $U \cap B_\alpha \in \termvalue{\alpha}$.
\end{proof}

\begin{prop} \em\label{prop:cvalimplicaprobigualaum}
Given a formula $\alpha$ and a stochastic valuation $V$,
$$\entc\alpha \quad \text{ implies } \quad \Prob_{V}(\alpha)=1.$$
\end{prop}
\begin{proof}\ \\
Assume $\entc\alpha$. Then, 
$$\termvalue{\alpha}=\wp B_\alpha.$$ Indeed it is immediate that
$\termvalue{\alpha}\subseteq \wp B_\alpha$.
For the other direction, let $U \subseteq B_\alpha$. Pick a valuation $v$ such that
$v \cap B_\alpha =U$. Then, $U \in \termvalue{\alpha}$ since $v \satc \alpha$.\\[1mm]
Hence
$$\Prob_{V}(\alpha)=\sum_{U \in \termvalue{\alpha}} \Prob(V_{B_\alpha}=U)=\sum_{U \subseteq B_\alpha} \Prob(V_{B_\alpha}=U)= 1.$$
\end{proof}

With these results in hand we are ready to show that every stochastic valuation assigns probabilities to formulas fulfilling Adams' principles. 

\begin{pth}\em\label{th:probassV}
Let $V$ be a stochastic valuation. Then, $\PSV{V}=\Prob_{V}$ is a probability assignment.
\end{pth}
\begin{proof}\ \\
Indeed, all the properties of probability assignments are satisfied:\\[2mm]
P1~Direct from the fact that $V_{B_\alpha}$ is a probability distribution for every $\alpha$.\\[2mm]
P2~Follows immediately from Proposition~\ref{prop:cvalimplicaprobigualaum}.\\[2mm]
P3.~This fact was proved in Proposition~\ref{prop:monprobs}.\\[2mm]
P4~Assume that $\entc \lneg (\beta \lconj \alpha)$. Then, there is no valuation $v$ such that $v \satc \beta$ and $v \satc \alpha$.
Hence,
\begin{align*}
\Prob_{V}(\beta \ldisj \alpha) = &\;  \displaystyle \sum_{U \in \termvalue{\beta \ldisj \alpha}} \Prob(V_{B_\beta \cup B_\alpha}=U) \displaybreak[1] \\[2mm]
                                             =&  \;\displaystyle \sum_{U\in \{v \cap (B_\beta \cup B_\alpha): v \satc \beta \ldisj \alpha\}} \Prob(V_{B_\beta \cup B_\alpha}=U) \displaybreak[1] \\[2mm]
                                             =& \; \displaystyle \sum_{U \in \{v \cap (B_\beta \cup B_\alpha): v \satc \beta \text{ or } v \satc \alpha\}} \Prob(V_{B_\beta \cup B_\alpha}=U) \displaybreak[1] \\[2mm]
                                              =& \;  \displaystyle \sum_{U \in \{v \cap (B_\beta \cup B_\alpha): v \satc \beta\}} \Prob(V_{B_\beta \cup B_\alpha}=U)
                                              	+  \dummy \displaybreak[1] \\[2mm]
                                            &\hspace*{19mm}  \displaystyle \sum_{U \in \{v \cap (B_\beta \cup B_\alpha): v \satc \alpha\}} \Prob(V_{B_\beta \cup B_\alpha}=U)  \displaybreak[1] \\[2mm]
  =& \;  \displaystyle \sum_{U \in \{v \cap B_\beta: v \satc \beta\}} \Prob(V_{B_\beta}=U)
 						+  \dummy \displaybreak[1] \\[2mm]
                                            &\hspace*{19mm} 
                                                   \displaystyle \sum_{U \in \{v \cap B_\alpha: v \satc \alpha\}} \Prob(V_{B_\alpha}=U) & (*) \displaybreak[1] \\[2mm]
                                          =& \;  \displaystyle \sum_{U \in \termvalue{\beta}} \Prob(V_{B_\beta}=U) +
                                                     \displaystyle \sum_{U \in \termvalue{\alpha}} \Prob(V_{B_\alpha}=U) \displaybreak[1] \\[2mm]
                                          =& \; \Prob_{V}(\beta) + \Prob_{V}(\alpha)                                                               
\end{align*}
where $(*)$ follows by Proposition~\ref{prop:mardisj}.
\end{proof}

We now show the converse result: each probability assignment induces a stochastic valuation giving back the original assignment. 
To this end, we first spell-out the family of finite-dimensional probabilistic valuations induced by a probability assignment and show that it fulfils the marginal condition. 
Afterwards, the envisaged stochastic valuation is obtained using Kolmogorov's existence theorem.

Given a probability assignment $P$, let
$$\fdd^P = \{\fdd^P_A = U \mapsto P(\phi_A^U): \wp A \to [0,1]\}_{A \in \pfwp B}.$$

\begin{prop} \em
Let $P$ be a probability assignment. There exists a unique stochastic valuation 
$$\SVP{P}$$
such that  
$\displaystyle \Prob(\SVP{P}_A=U)=P(\phi_A^U).$ 
\end{prop}
\begin{proof}\ \\
(1)~Each $\fdd^P_A$ is a finite-dimensional probabilitistic valuation: \\[2mm] 
(a) $\fdd^P_A(U)\in [0,1]$. Follows immediately from P1.\\[2mm]
(b) $\sum_{U \subseteq A} \fdd^P_A(U) = 1$. Indeed: 
\begin{align*}
\displaystyle \sum_{U \subseteq A} \fdd^P_A(U) =& \; \displaystyle \sum_{U \subseteq A} P(\phi_A^U)  \displaybreak[1] \\[2mm]
                                                                             =& \; \displaystyle P\left(\bigvee_{U \subseteq A} \phi_A^U\right) & (*) \displaybreak[1] \\[2mm]
                                                                             =&  \;  1& (**)
\end{align*}
where $(*)$ follows from Proposition~\ref{prop:contUdif} and P4 and $(**)$ follows from Proposition~\ref{prop:disjtottaut} and P2.\\[2mm]
(c)~Additivity is trivial since we are dealing with a measure over a finite set of outcomes.\\[2mm]
(2)~The family $\fdd^P$ fulfils the marginal condition. Assume that $A' \subseteq A$ and $U' \subseteq A'$. Then, 
\begin{align*}
\displaystyle \sum_{\text{\shortstack[c]
{$U \subseteq A$\\
$U \cap A'=U'$
}
}} \fdd^P_A(U)         =& \; \displaystyle \sum_{\text{\shortstack[c]
{$U \subseteq A$\\
$U \cap A'=U'$
}
}} P(\phi_A^U) \displaybreak[1] \\[2mm]
                              =& \; \displaystyle P\left(\bigvee_{\text{\shortstack[c]
{$U \subseteq A$\\
$U \cap A'=U'$
}
}}  \phi_A^U\right) & (*)\displaybreak[1] \\[2mm]
                             =& \; P(\phi_{A'}^{U'}) & (**)\displaybreak[1] \\[2mm]
                             =& \;  \fdd^P_{A'}(U')                           
\end{align*}
where $(*)$ follows from Proposition~\ref{prop:contUdif} and P4 and $(**)$ follows from Proposition~\ref{prop:marglog} and P3.\\[2mm]
Hence, using Kolmogorov's existence theorem, there exists a unique stochastic valuation having these finite-dimensional distributions. Let $\SVP{P}$ be this stochastic valuation.
\end{proof}

We now proceed to show that $\SVP{P}$ induces back the original probability assignment $P$
and, conversely, that a stochastic valuation $V$ induces the probability assignment $\PSV{V}$ that gives back the original $V$. To this end, we need the following auxiliary result.

\begin{prop}\em \label{prop:normdisj}
Let $U' \subseteq A' \subseteq A \subseteq B$ and $\beta$ a formula in $L$. Then
$$\entc \left(\bigvee_{U \in \{v \cap B_\beta: v \satc \beta\}} \phi_{B_\beta}^U\right) \leqv \beta.$$
\end{prop}
\begin{proof}\ \\
Let $v$ be a valuation. \\[1mm]
$(\to)$ Assume that
$$v \satc \left(\bigvee_{U \in \{v \cap B_\beta: v \satc \beta\}} \phi_{B_\beta}^U\right).$$
Let $U \in \{v \cap B_\beta: v \satc \beta\}$ be such that 
$$v \satc \phi_{B_\beta}^U.$$
Then, there is $v'$ such that $U=v' \cap B_\beta$ and $v' \satc \beta$. Hence
$$v \satc \phi_{B_\beta}^{v' \cap B_\beta} \quad (\dag).$$
We now show that $v' \cap B_\beta = v \cap B_\beta$. Let $B_j \in v' \cap B_\beta$. Then, $v \satc B_j$ by $(\dag)$ and so
 $B_j \in v \cap B_\beta$. For the other direction let $B_j \in v \cap B_\beta$. By $(\dag)$, $B_j \in v'$. Hence
 $B_j \in v' \cap B_\beta$. \\[1mm]
 Therefore,
 $$U=v \cap B_\beta$$
 and so $v \satc \beta$.\\[2mm]
 $(\from)$ Assume that $v' \satc \beta$. Observe that $v' \satc \phi_{B_\beta}^{v' \cap B_\beta}$.
Then, $v' \cap B_\beta \in \{v \cap B_\beta: v \satc \beta\}$ and so the thesis follows.
\end{proof}

\begin{pth}\em
Let $V$ be a stochastic valuation and $P$ a probability assignment. 
Then,
$$\SVP{\PSV{V}}=V \quad \text{and} \quad \PSV{\SVP{P}}=P.$$
\end{pth}
\begin{proof}\ \\
Let $A \in \pfwp B$ and $U \subseteq A$. Observe that we have:
\begin{align*} \displaystyle
\Prob(\SVP{\PSV{V}}_A=U) =& \; \PSV{V}(\phi_A^U) \displaybreak[1] \\[2mm]
                                     =& \; \Prob_V(\phi_A^U) \displaybreak[1] \\[2mm]
                                     =&  \; \displaystyle \sum_{U' \in \termvalue{\phi_A^U}} \Prob(V_{B_{\phi_A^U}}=U')\displaybreak[1] \\[2mm]
                                     =& \; \Prob(V_{B_{\phi_A^U}}=U) \displaybreak[1] \\[2mm]
                                    =& \; \Prob(V_{A}=U).
\end{align*}
Therefore, the stochastic valuations $\SVP{\PSV{V}}$ and $V$ have the same finite-dimensional probabilistic valuations and, so, by the Kolmogorov's existence theorem, they are equivalent.\\[2mm]
Moreover, let $\beta\in L$. Then:
\begin{align*}
\PSV{\SVP{P}}(\beta) =\Prob_{\SVP{P}}(\beta) =& \; \displaystyle \sum_{U \in \termvalue{\beta}} \Prob(\SVP{P}_{B_\beta}=U)\displaybreak[1] \\[2mm]
                                     =& \;  \ds \sum_{U \in \{v \cap B_\beta: v \satc \beta\}} P(\phi_{B_\beta}^U)\displaybreak[1] \\[2mm]
                                     =&  \; \ds P\left(\bigvee_{U \in \{v \cap B_\beta: v \satc \beta\}} \phi_{B_\beta}^U\right) &
                                     														 (*)\displaybreak[1] \\[2mm]
                                     =& \; P(\beta) & (**)                                     
\end{align*}
where $(*)$ follows from Proposition~\ref{prop:contUdif} and P4 and $(**)$ is a consequence of Proposition~\ref{prop:normdisj} and P3.
\end{proof}

In short, there is a strict Galois connection between stochastic valuations and probability assignments to (classical) formulas.
Therefore, we can freely choose to assign probabilities to formulas or to valuations. 
In the remainder of this paper we adopt the latter approach, using stochastic valuations for the purpose.

\section{Probabilistic entailment}\label{sec:pent}

In this section, we compare
the entailment in $\CPL$ (classical  propositional logic with valuations as semantics)
with
the probabilistic entailments that we are able to define in $\svPL$ (a variant of $\CPL$ with the same language but adopting stochastic valuations as semantics). 
The key result of this section is the collapse of these probabilistic entailments into the classical entailment.

It is possible to define in $\svPL$ a family $\entpq{p}{q}$ of probabilistic entailments depending on the minimal probability $p$ required from the hypotheses in order to obtain the conclusion with at least probability $q$. 
To this end, we first define satisfaction by a stochastic valuation of a formula with a minimal probability $p$.

Let $V$ be a stochastic valuation, $\alpha \in L$ and $p \in [0,1]$. We say that $\alpha$ is {\it $p$-satisfied} by  $V$, written
$$V \satp{p} \alpha,$$ 
whenever $\Prob_{V}(\alpha) \geq p$. 
That is, a formula is $p$-satisfied by $V$ whenever its probability under $V$ is at least $p$.

According to Hailperin~\cite{hai:96,hai:11}, 
given $\Delta \cup\{\alpha\} \subseteq L$ and $p,q \in [0,1]$,
one would say that $\Delta$ $pq$-entails $\alpha$, written here
$$\Delta \hentpq{p}{q} \alpha,$$
whenever, for every stochastic valuation $V$,
$$\text{if } V \satp{p} \delta \text{ for every } \delta \in \Delta \text{ then } V \satp{q} \alpha.$$
That is, if the probability under $V$ of each hypothesis is at least $p$ then the probability under $V$ of the conclusion is at least $q$.

However, we find this definition wanting since $\hentpq{p}{q}$ does not enjoy the following desirable property:
$$\delta_1,\delta_2 \hentpq{p}{q} \alpha \quad\text{iff}\quad \delta_1\lconj\delta_2 \hentpq{p}{q} \alpha.$$
Indeed, for instance,
$$B_1 \lconj (\lneg B_1) \hentpq{\frac12}{\frac14} \lfalsum$$
while
$$B_1 ,\lneg B_1 \not\hentpq{\frac12}{\frac14} \lfalsum.$$
The former holds vacuously because $\tv{B_1 \lconj (\lneg B_1)}=\emptyset$.
Concerning the latter, observe that 
it is easy to find a stochastic valuation $V$ such that 
$\Prob_V(B_1)=\Prob(\lneg B_1)=\frac12$ 
and note that every stochastic valuation assigns probability zero to $\lfalsum$
since $\tv\lfalsum=\emptyset$. In fact, Hailperin's concept requires that both $\delta_1$ and $\delta_2$ have probability greater than or equal to $p$ of being true when one should instead require that the probability of them being simultaneously true is greater than or equal to $p$.
 
In order to overcome this difficulty, we propose to use the following notion of probabilistic entailment where, as usual, for any finite set $\Phi$ of formulas, 
we write $\bigwedge \Phi$ for the conjunction of the formulas in $\Phi$, with $\bigwedge\emptyset$ standing for $\lverum$.

Let $\Delta \cup\{\alpha\} \subseteq L$ and $p,q \in (0,1]$ such that $p \geq q$.
We say that $\Delta$ {\it $pq$-entails} $\alpha$, written
$$\Delta \entpq{p}{q} \alpha,$$
whenever there is a finite subset $\Phi$ of $\Delta$ such that,
for every stochastic valuation $V$,
$$\text{if } V \satp{p} \bigwedge \Phi \text{ then } V \satp{q} \alpha.$$
Clearly,
$$\Delta \entpq{p}{q} \alpha \quad \text{ iff } \quad 
	 \exists\, \Phi {\in} \fwp\Delta: \bigwedge\Phi  \hentpq{p}{q}\alpha.$$
Thus, when $\Delta$ is a singleton or the empty set the two definitions coincide.
	 
Note that the requirement $q > 0$ is well justified because $ \entpq{p}{0}$ is trivial. Indeed, every formula is $p0$-entailed by any set of hypotheses since $\Prob_V(\alpha)\geq0$ for every stochastic valuation $V$ and formula $\alpha$.

Observe also that the requirement $p \geq q$ is essential since otherwise the induced {\it $pq$-entailment operator}
$$\Delta \mapsto \Delta^{\entpq{p}{q}}=\{\alpha\in L: \Delta \entpq{p}{q} \alpha\}: \wp L \to \wp L$$
would not be extensive. Indeed, for instance,
$$B_1 \not\entpq{\frac14}{\frac34} B_1.$$

It is straightforward to verify that the $pq$-entailment operator is extensive if $p \geq q$, as well as monotonic for arbitrary $p$ and $q$. 
On the other hand, it is not clear from the definition if it is idempotent.
In fact, each $pq$-entailment operator is indeed idempotent but the proof is not trivial.
Idempotence is not used on the way to the collapsing theorem at the end of this section.
Moreover, it follows immediately from that theorem.
Therefore, we refrain from attempting at this point to prove the idempotence of each $pq$-entailment operator.
Observe also that it follows directly from its definition that each operator is compact.


The aim now is to compare the probabilistic entailments of $\svPL$ with the entailment of $\CPL$.  

To this end, we need to explain how a classical valuation canonically induces a stochastic valuation.
Given a valuation $v$, consider the family of maps
$$\fdd^v =\{\fdd^v_A: \wp A \to [0,1]\}_{A\in\pfwp B}$$
where each map is as follows:
$$\fdd^v_A(U) = \begin{cases}
       1 & U=v \cap A\\
       0 & \text{ otherwise}.
       \end{cases}$$

\begin{prop} \em \label{prop:inducedSV}
Given a valuation $v$, $\fdd^v$ is a consistent family of finite dimensional probability valuations.
\end{prop}
\begin{proof}\ \\
Since it is straightforward to check that each $\fdd^v_A$ is a probability measure over $A$,
we focus on showing that $\fdd^v$ fulfils the marginal condition. 
Let $A \in \pfwp B$, $A' \in \pwp A$ and $U' \in \wp A'$. Then, consider two cases:\\
(i)~$v \cap A' = U'$. Observe that $v \cap A\subseteq A$ and $v \cap A \cap A'=v\cap A'=U'$. Hence
$$
\displaystyle \sum_{\text{\shortstack[c]
{$U \subseteq A$\\
$U \cap A'=U'$
}
}} 
 \fdd^v_A(U) = 1 = \fdd_{A'}(U').
$$
(ii)~$v \cap A'\neq U'$. Then, $v \cap A \cap A' =v \cap A' \neq U'$ and so 
$$
\displaystyle \sum_{\text{\shortstack[c]
{$U \subseteq A$\\
$U \cap A'=U'$
}
}} 
 \fdd^v_A(U) = 0 = \fdd_{A'}(U').
$$
\end{proof}

Therefore, using Kolmogorov's existence theorem, there exists a unique stochastic valuation $V^v$ inducing the finite-dimensional probability valuations in $\fdd^v$.
We say that $V^v$ is the {\it stochastic valuation induced by} $v$. Observe that
$$\Prob(V^v_A=U)=\fdd^v_A(U)$$
for each $A \in \pfwp B$ and $U \subseteq A$.
The next result establishes the envisaged relationship between satisfaction by a valuation
and satisfaction by its induced stochastic valuation.

\begin{prop} \em \label{prop:satvssatp}
Given a formula $\alpha \in L$,  a valuation $v$ and $p \in\; (0,1]$
$$v \satc \alpha \quad \text{ iff } \quad V^v\sat_p\alpha.$$
\end{prop}
\begin{proof}\ \\
($\to$)~Assume that $v \satc \alpha$. Then
$$v\cap B_\alpha \in \termvalue{\alpha}$$
by definition of $\termvalue{\alpha}$. Hence, 
$$\Prob_{V^v}(\alpha)=\sum_{U \in \termvalue{\alpha}} \Prob(V^v_{B_\alpha} =U) = \Prob(V^v_{B_\alpha} =v\cap B_\alpha)=1 \geq p,$$
by definition of $V^v$. So, $V^v\sat_p\alpha$.\\[2mm]
($\from$)~Assume that $V^v\sat_p\alpha$. Then, $\Prob_{V^v}(\alpha)\geq p$.
Hence,
$$\sum_{U \in \termvalue{\alpha}} \Prob(V^v_{B_\alpha}=U) \geq p>0.$$
Observe that $\Prob(V^v_{B_\alpha}=U)=0$ for every $U \neq v \cap B_\alpha$
and $\Prob(V^v_{B_\alpha} =v \cap B_\alpha)=1$.
Therefore, $v \cap B_\alpha \in \termvalue{\alpha}$.
Thus, $v \satc \alpha$.
\end{proof}

We now proceed with the investigation of the relationship between the probabilistic entailments and the classical entailment. To this end, we need the following auxiliary result.

\begin{prop} \em \label{prop:clasimpprob}
Given formulas $\delta$ and $\alpha$ with $\delta \entc \alpha$ and $p,q \in (0,1]$ with $p \geq q$,
$$\text{if }V \satp{p} \delta\text{ then }V \satp{q} \alpha$$
for every stochastic valuation $V$.
\end{prop}
\begin{proof}\ \\
Let $V$ be a stochastic valuation such that $V \satp{p} \delta$. 
Hence, $$\Prob_{V}(\delta) \geq p.$$
So, by Proposition~\ref{prop:monprobs},
$$\Prob_{V}(\alpha) \geq \Prob_{V}(\delta) \geq p \geq q.$$
Thus, $V \satp{q} \alpha$.
\end{proof}

The next result shows that the probabilistic entailments collapse into the classical entailment.

\begin{pth} \label{th:collapse}\em
Given a set of formulas $\Delta$, a formula $\alpha$ and $p,q \in (0,1]$ with $p \geq q$,
$$\Delta \entc \alpha \quad \text{ iff } \quad \Delta \entpq{p}{q}\alpha.$$
\end{pth}
\begin{proof}\\
$(\to)$ Assume that $\Delta \entc \alpha$, and let $\Phi$ be a finite subset of $\Delta$ such that $\Phi \entc \alpha$. 
Hence, $\bigwedge \Phi \entc \alpha$.
Thus, by Proposition~\ref{prop:clasimpprob}, if $V \satp{p} \bigwedge \Phi$ then $V \satp{q} \alpha$, for every stochastic valuation $V$. Therefore, by definition, $\Delta \entpq{p}{q}\alpha$.\\[2mm]
$(\from)$ Assume that $\Delta \entpq{p}{q}\alpha$, and let $\Phi$ be a finite subset of $\Delta$ such that,
for every stochastic valuation $V$,
$$\text{if } V \satp{p} \bigwedge \Phi \text{ then } V \satp{q} \alpha.$$
Let $v$ be a valuation such that $v \satc \delta$ for each $\delta \in \Delta$. Then, $v\satc \bigwedge \Phi$. Observe that $p >0$. Then, $V^v \satp{p} \bigwedge \Phi$, by Proposition~\ref{prop:satvssatp}. So, $V^v \satp{q} \alpha$ because $\Delta \entpq{p}{q}\alpha$. Thus, again by Proposition~\ref{prop:satvssatp},
$v \satc \alpha$ since $q > 0$.
\end{proof}

In conclusion, since nothing is gained in terms of entailment by enriching the semantics of the classical proposicional logic with the means for assigning probabilities to formulas, it is necessary to extend the language with probabilistic constructs. That is precisely the objective of the rest of this paper.

\section{Probabilistic propositional logic}\label{sec:PPL}

The objective of this section is to define an enrichment of $\CPL$ that captures the probabilistic nature of the semantics provided by stochastic valuations. 
The idea is to add as little as possible to the propositional language $L$. 
It turns out that adding a symbolic construct allowing the constraining of the probability of a formula is enough.

Before proceeding with the presentation of the envisaged probabilistic propositional logic ($\PPL$), we need to adopt some notation concerning the first-order theory of real closed ordered fields ($\RCOF$),
having in mind the use of its terms for denoting probabilities and other quantities.

Recall that the first-order signature of $\RCOF$ contains the constants $\num0$ and $\num1$, the unary function symbol $-$, the binary function symbols $+$ and $\times$, and the binary predicate symbols $=$ and $<$. 
We take the set $X = X_\nats \cup X_L$,
where $X_\nats=\{x_k: k \in \nats\}$ and $X_L=\{x_\alpha: \alpha \in L\}$,
as the set of variables. 
In the sequel, by $T_\RCOF$ we mean the set of terms in $\RCOF$ that do not use variables in $X_L$.
As we shall see, the variables in $X_L$ become handy in the proposed axiomatization $\PPL$,
for representing within the language of $\RCOF$ the probability of $\alpha$.

As usual, we may write $t_1 \leq t_2$ for $(t_1 < t_2) \ldisj (t_1=t_2)$,
$t_1\,t_2$ for $t_1 \times t_2$ and $t^n$ for 
$$\underbrace{t \times \dots \times t}_{n \text{ times}}.$$
Furthermore, we also use the following abbreviations for any given $m\in\pnats$ and $n\in\nats$:
\begin{itemize}

\item $\num m$ for $\underbrace{\num1+\dots +\num1}_{\text{addition of } m \text{ units}}$;

\item $\inv{\num m}$ for the unique $z$ such that $\num m \times z = \num1$;

\item $\ds\frac{\num n}{\num m}$ for $\inv{\num m} \times \num n$.

\end{itemize}
The last two abbreviations might be extended to other terms, but we need them only for numerals. 
For the sake of simplicity, we do not notationally distinguish between a natural number and the corresponding numeral.

In order to avoid confusion with the other notions of satisfaction used herein, we adopt $\satfo$ for denoting satisfaction in first-order logic (over the language of $\RCOF$).

Recall also that the theory $\RCOF$ is decidable~\cite{tar:51}. This fact will be put to good use in the axiomatization for $\PPL$ (presented at the end of this section) and, further on (in Section~\ref{sec:cons}), when proving the decidability of $\PPL$.
Furthermore, every model of $\RCOF$ satisfies the theorems and only the theorems of $\RCOF$ (Corollary~3.3.16 in~\cite{mar:02}).
We shall take advantage of this result in the semantics of $\PPL$ for adopting the ordered field $\reals$ of the real numbers as the model of $\RCOF$.

\subsection{Language}

The language $L_\PPL$ of the propositional probability logic $\PPL$ is inductively defined as follows:

\begin{itemize}
\item $\inte{\alpha} \mycirc p \in L_\PPL$ where $\alpha \in L$,  $p \in T_\RCOF$ and $\mycirc \in \{=,<\}$;
\item $\varphi_1 \limp \varphi_2 \in L_\PPL$ whenever $\varphi_1,\varphi_2 \in L_\PPL$.
\end{itemize}

Propositional abbreviations can be introduced as usual. For instance,
$$\lneg \varphi \;\text{ for }\; \varphi \limp (\inte{\lverum} < 1)$$
and similarly for $\lconj$, $\ldisj$ and $\leqv$.
Comparison abbreviations also become handy. For instance,
$$\inte{\alpha} \leq p \;\text{ for }\; (\inte{\alpha} = p) \ldisj (\inte{\alpha} <  p)$$
and
$$\inte{\alpha} \geq p \;\text{ for }\; \lneg (\inte{\alpha} <  p).$$

\subsection{Semantics}

Given a term $t$ and an assignment $\rho: X \to \reals$, 
we write $t^{\reals\rho}$ for the denotation of term $t$ in $\reals$ for $\rho$. 
When $t$ does not contain variables we may use $t^\reals$ for the denotation of $t$ in $\reals$.

Let $V$ be a stochastic valuation and $\rho$ an assignment. 
{\it Satisfaction of formulas} by $V$ and $\rho$  is inductively defined as follows:
\begin{itemize}
\item $V \rho \sat \inte{\alpha} \mycirc p$ whenever $\Prob_V(\alpha) \mycirc p^{\reals\rho}$;
\item $V \rho \sat \varphi_1 \limp \varphi_2$ whenever $V\rho \not \sat \varphi_1$ or $V \rho \sat  \varphi_2$.
\end{itemize}
We may omit the reference to the assignment $\rho$ whenever the formula does not include variables.

Let $\Gamma \subseteq L_\PPL$ and $\varphi \in L_\PPL$. We say that $\Gamma$ {\it entails} $\varphi$, written $\Gamma \ent \varphi$,
whenever, for every stochastic valuation $V$ and assignment $\rho$,
if $V \rho \sat \gamma$ for each $\gamma \in \Gamma$ then $V \rho \sat \varphi$.
As expected, $\varphi$ is said to be {\it valid} when $\ent \varphi$.

Observe that entailment in $\PPL$ is not compact. Indeed, since $\reals$ is Archimedean,
$$\left\{\inte{\alpha} \leq \frac 1n : n \in \nats\right\} \ent \inte{\alpha}=0.$$
However, there is no finite subset $\Psi$ of $\{\inte{\alpha} \leq \frac 1n : n \in \nats\}$ such that
$$\Psi \ent \inte{\alpha}=0.$$

\subsection{Calculus}

The $\PPL$ calculus combines propositional reasoning with $\RCOF$ reasoning. 
We intend to use the $\RCOF$ reasoning to a minimum, namely to prove assertions like  
$$\left(\inte{\alpha_1} \mycirc_1 p_1 \lconj \cdots \lconj   \inte{\alpha_k} \mycirc_k p_k\right) \limp \inte{\alpha_{k+1}} \mycirc_{k+1} p_{k+1}.$$

To this end, we represent in $\RCOF$ the probability $\inte{\alpha}$ of each propositional formula $\alpha$ by variable $x_\alpha$ and impose conditions on that variable that effect the properties of the probability.

Recall that the probability of a formula $\alpha$ is the sum of the probabilities of the $B_\alpha$-valuations that satisfy the formula and that there is a disjunctive normal form of $\alpha$ where each disjunct can be seen as identifying a $B_\alpha$-valuation that satisfies the formula. Hence, for calculating the probability of $\alpha$ it is enough to sum the probabilities of each such disjunct. 

As we proceed to explain, 
we collect these conditions in a formula of $\RCOF$.
We say that 
$\Lambda = \{\alpha_{11}, \ldots, \alpha_{1m_1}, \ldots, \alpha_{k1}, \ldots, \alpha_{km_{k}}\} \subset L$ 
is an {\it adequate set of DNF-conjuncts}
for $\{\alpha_1,\ldots,\alpha_{k}\} \subset L$ whenever 
\begin{enumerate}
\item  $\ds B_{\alpha_{11}}=\dots = B_{\alpha_{km_{k}}} = 
			B_{\alpha_1} \cup \dots \cup B_{\alpha_{k}}= B_\Lambda$; 
\item each $\alpha_{j\ell}$ is a conjunction of literals;
\item $\entc \lneg(\alpha_{j\ell} \lconj \alpha_{j\ell'})$ for $1\leq \ell \neq \ell' \leq m_j$;
\item $\entc \alpha_j \leqv \bigvee_{\ell=1}^{m_j} \alpha_{j\ell}$ 
for each $j=1,\ldots,k$.
\end{enumerate}
Observe that clauses 2 and 4 ensure that $\bigvee_{\ell=1}^{m_j} \alpha_{j\ell}$ is a disjunctive normal form of $\alpha_j$. 
Moreover, clause 3 guarantees that there are not redundant disjuncts in the disjunctive normal form of each $\alpha_j$. 
Given such set $\Lambda$ of adequate DNF-conjuncts for $\alpha_1,\ldots,\alpha_{k}$,
we use the abbreviation $$Q^{\alpha_1,\ldots,\alpha_{k}}_{\alpha_{11},\ldots,\alpha_{km_{k}}}$$
for the $\RCOF$ formula
$$\left(\bigwedge_{U\subseteq B_\Lambda} 0 \leq x_{\phi_{B_\Lambda}^U} \leq 1\right)\lconj
	\left(\sum_{U\subseteq B_\Lambda} x_{\phi_{B_\Lambda}^U}=1\right)\lconj
	\left(\bigwedge_{j=1}^{k}\left(x_{\alpha_j}=\sum_{\ell=1}^{m_j} x_{\alpha_{j\ell}}\right)\right).$$ 
Recall that each propositional formula $\phi_{B_\Lambda}^U$ identifies the $B_\Lambda$-valuation that assigns true to the propositional symbols in $U$ and assigns false to the propositional symbols in $B_\Lambda\setminus U$. Hence, 
$$0 \leq x_{\phi_{B_\Lambda}^U} \leq 1$$
imposes that the probability of each $B_\Lambda$-valuation should be in the interval $[0,1]$ and 
$$\sum_{U\subseteq B_\Lambda} x_{\phi_{B_\Lambda}^U}=1$$
imposes in $\RCOF$ that the sum of the probabilities of all $B_\Lambda$-valuations is $1$.
The conjunct
$$x_{\alpha_j}=\sum_{\ell=1}^{m_j} x_{\alpha_{j\ell}}$$
imposes that the probability of $\alpha_j$ is the sum of the probabilities of the valuations that satisfy the formula.

The calculus for $\PPL$ is an extension of the classical propositional calculus containing the following axioms and rules:

\begin{itemize}

\item[$\TAUT$] $\ds\lrule{}{\varphi}$\\[2mm]
provided that
$\varphi$ is a tautological formula;

\item[$\RR$] $\ds\lrule{}
{\left(\inte{\alpha_1} \mycirc_1 p_1 \lconj \cdots \lconj   \inte{\alpha_k} \mycirc_k p_k\right) \limp
									 \inte{\alpha_{k+1}} \mycirc_{k+1} p_{k+1}}$ \\[3mm]
provided that 
there is an adequate set 
$$\{\alpha_{11}, \ldots, \alpha_{1m_1}, \ldots, \alpha_{(k+1) 1}, \ldots, \alpha_{(k+1)m_{k+1}}\}$$
of DNF-conjuncts for $\{\alpha_1,\dots,\alpha_{k+1}\}$ such that \\[4mm]
\hspace*\fill $\ds\bigforall
		\left(\left(Q^{\alpha_1,\ldots,\alpha_{k+1}}_{\alpha_{11},\ldots,\alpha_{k+1m_{k+1}}} \lconj \bigwedge_{j=1}^k \left(x_{\alpha_{j}} \mycirc_j p_j\right)\right)
					\limp \left(x_{\alpha_{k+1}} \mycirc_{k+1} p_{k+1}\right)\right)$\\[2mm]
is a theorem of $\RCOF$;

\item[$\MP$] $\ds\lrule{\begin{array}{l}
							\varphi_1 \quad
							\varphi_1 \limp\varphi_2
					\end{array}}
					{\varphi_2}$.
\end{itemize}

Axioms $\TAUT$ and rule $\MP$ extend the propositional reasoning to formulas in $L_\PPL$. Axioms $\RR$ import to $\PPL$ all we need from $\RCOF$.

\subsection{Examples}\label{subsec:pplexamples}

Consider the derivation in Figure~\ref{fig:probleq1} for establishing that
$$\der \inte{\alpha} \leq 1$$
holds for arbitrary $\alpha \in L$. The use of RR axioms can be involved. We explain their use only for
obtaining $\inte{\lverum}=1$ (step 2 of the derivation). Assume that $\lverum$ is an abbreviation of $B_1 \ldisj (\lneg B_1)$. Then
the following formula is in $\RCOF$:
$$\ds\bigforall
		\left(Q^{\lverum}_{B_1,\lneg B_1}
					\limp \left(x_{\lverum} = 1\right)\right)$$ 
where $Q^{\lverum}_{B_1,\lneg B_1}$ is 
$$\begin{array}{c}
(0 \leq x_{B_1} \leq 1) \lconj  (0 \leq x_{\lneg B_1} \leq 1)\\ \lconj\\
  x_{B_1} + x_{\lneg B_1}=1\\ 
 \lconj \\
 x_{\lverum} = x_{B_1} + x_{\lneg B_1},
\end{array}$$
and, so, $\inte{\lverum} = 1$ is obtained by $\RR$.

\begin{figure}[t]
$$\begin{array}{rlr}
	1 & \inte{(\alpha \limp \lverum)}=1 &\RR\\[2mm]			              
	2 & \inte{\lverum} =1	& \RR\\[2mm]
	3 &  (\inte{(\alpha \limp \lverum)}=1) \limp ( (\inte{\lverum} =1) \limp ( (\inte{(\alpha \limp \lverum)}=1) \lconj (\inte{\lverum} =1))) &  \TAUT\\[2mm]
	4 &  (\inte{\lverum} =1) \limp ( (\inte{(\alpha \limp \lverum)}=1) \lconj (\inte{\lverum} =1)) & \MP \; 1,3	\\[2mm]
	5 &  (\inte{(\alpha \limp \lverum)}=1) \lconj (\inte{\lverum} =1) & \MP \; 2,4	\\[2mm]
	6 & ((\inte{(\alpha \limp \lverum)}=1) \lconj (\inte{\lverum} =1)) \limp  (\inte{\alpha} \leq 1)	& \RR\\[1mm]
	7 & \inte{\alpha} \leq 1 & \MP\;  5,6
\end{array}$$\vspace*{-6mm}
\caption{$\dummy\der \inte{\alpha} \leq 1$.}\label{fig:probleq1}
\end{figure}

Next, let us see how we can express and derive (binary) additivity, i.e., how we can 
establish that
$$\inte{\alpha}=x_1,\inte{\beta}=x_2,\inte{(\alpha \lconj \beta)}=x_3
\der {\inte{(\alpha \ldisj \beta}) = x_1 + x_2 - x_3}$$
holds for arbitrary $\alpha,\beta \in L$. This follows immediately using $\RR$ since, letting $\{\alpha_{11},\ldots,\alpha_{4m_4}\}$ be an adequate set of DNF-conjuncts for $\{\alpha_1,\ldots,\alpha_4\}$ where
\begin{itemize}
\item $\alpha_1$ is $\alpha$;
\item $\alpha_2$ is $\beta$;
\item $\alpha_3$ is  $\alpha \lconj \beta$ and each $\alpha_{3\ell}$ is a common disjunct in the disjunctive normal form of $\alpha_1$ and $\alpha_2$;
\item $\alpha_4$ is $\alpha \ldisj \beta$ and each $\alpha_{4\ell}$ is a disjunct in the disjunctive normal form of $\alpha_1$ and $\alpha_2$ with no repetitions;
\end{itemize}
the following formula
$$\ds\bigforall
		\left(\left(Q^{\alpha_1,\ldots,\alpha_{4}}_{\alpha_{11},\ldots,\alpha_{4m_{4}}} \lconj \bigwedge_{j=1}^3 \left(x_{\alpha_{j}} = x_j\right)\right)
					\limp \left(x_{\alpha_{4}} = x_1+x_2-x_3\right)\right)$$ 
is a theorem of $\RCOF$.

The marginal condition is also expressible and derivable. For instance, let $A=\{B_1,B_2\}$, $A'=\{B_1\}$ and $U'=\{B_1\}$.
We present in Figure~\ref{fig:marg} a derivation for showing that
$$\inte{(B_1 \lconj  (\lneg B_2))} = x_{1},\inte{(B_1 \lconj B_2)} = x_{2}  \der \inte{B_1} = x_1+x_2$$
holds.
\begin{figure}[t]
$$\begin{array}{rlr}
	1 & \inte{(B_1 \lconj  (\lneg B_2))} = x_{1} &\HYP\\[2mm]			              
	2 & \inte{(B_1 \lconj B_2)} = x_{2}	& \HYP\\[2mm]
	3 &  (\inte{(B_1 \lconj  (\lneg B_2))} = x_{1}) \limp ( (\inte{(B_1 \lconj B_2)} = x_{2}) \limp \dummy\\[1mm]
	     & \hspace*{7mm} ( (\inte{(B_1 \lconj  (\lneg B_2))} = x_{1}) \lconj (\inte{(B_1 \lconj B_2)} = x_{2}))) &  \TAUT\\[2mm]
	4 &  (\inte{(B_1 \lconj B_2)} = x_{2}) \limp \dummy\\[1mm]
	     & \hspace*{7mm} ( (\inte{(B_1 \lconj  (\lneg B_2))} = x_{1}) \lconj (\inte{(B_1 \lconj B_2)} = x_{2})) & \MP \; 1,3	\\[2mm]	
	5 &  (\inte{(B_1 \lconj  (\lneg B_2))} = x_{1}) \lconj (\inte{(B_1 \lconj B_2)} = x_{2}) & \MP \; 2,4	\\[2mm]	
	6 &  ((\inte{(B_1 \lconj  (\lneg B_2))} = x_{1}) \lconj (\inte{(B_1 \lconj B_2)} = x_{2})) \limp \dummy\\[1mm]
	& \hspace*{55mm}(\inte{B_1} = x_1+x_2) & \RR	\\[2mm]	
	7 &  \inte{B_1} = x_1+x_2 & \MP \; 5,6	\\[2mm]
\end{array}$$\vspace*{-6mm}
\caption{$\dummy \inte{(B_1 \lconj  (\lneg B_2))} = x_{1},\inte{(B_1 \lconj B_2)} = x_{2}  \der \inte{B_1} = x_1+x_2$.}\label{fig:marg}
\end{figure}

The next example shows how {\it modus ponens} for classical formulas is lifted to $\PPL$ formulas.
Observe that
$$\inte{\alpha_1} = 1, \inte{\alpha_1 \limp \alpha_2} = 1 \der \inte{\alpha_2} = 1$$
holds, as can be seen in Figure~\ref{fig:derMPPPL}.
\begin{figure}[t]
$$\begin{array}{rlr}
	1 & \inte{(\alpha_1)}=1 &\HYP\\[2mm]			              
	2 & \inte{(\alpha_1 \limp \alpha_2)}=1 & \HYP\\[2mm]
	3 &  (\inte{(\alpha_1)}=1) \limp ( (\inte{(\alpha_1 \limp \alpha_2)}=1) \limp \\[1mm]
	     & \hspace*{15mm}( (\inte{(\alpha_1)}=1) \lconj (\inte{(\alpha_1 \limp \alpha_2)}=1))) &  \TAUT\\[2mm]
	4 &  (\inte{(\alpha_1 \limp \alpha_2)}=1) \limp 
	     ( (\inte{(\alpha_1)}=1) \lconj (\inte{(\alpha_1 \limp \alpha_2)}=1)) & \MP \; 1,3	\\[2mm]
	5 &  (\inte{(\alpha_1)}=1) \lconj (\inte{(\alpha_1 \limp \alpha_2)}=1) & \MP \; 2,4	\\[2mm]
	6 & ((\inte{(\alpha_1)}=1) \lconj (\inte{(\alpha_1 \limp \alpha_2)}=1)) \limp  (\inte{\alpha_2} = 1)	& \RR\\[1mm]
	7 & \inte{\alpha_2} = 1 & \MP\;  5,6
\end{array}$$\vspace*{-6mm}
\caption{$\inte{\alpha_1} = 1, \inte{\alpha_1 \limp \alpha_2} = 1 \der \inte{\alpha_2} = 1$.}\label{fig:derMPPPL}
\end{figure}
Therefore, the rule 
$$\MP^* \quad \ds\lrule{\begin{array}{l}
							\inte{\alpha_1} = 1 \quad
							\inte{\alpha_1 \limp \alpha_2} = 1
					\end{array}}
					{\inte{\alpha_2} = 1}$$
is admissible in the $\PPL$ calculus.

In the same vein it is easy to show that the rule
$$\TAUT^* \quad \ds\lrule{  }{\inte{\alpha} = 1} \quad \text{provided that} \entc\alpha $$
is also admissible in the $\PPL$ calculus. Just observe that 
$$\{\phi^U_{B_\alpha}: U \subseteq B_\alpha\}$$
is an adequate set of $\DNF$-conjuncts for $\alpha$, 
since $\alpha$ is a tautology. Then, it is immediate to conclude that
$(\inte{\lverum}=1) \limp (\inte{\alpha} = 1)$ is an $\RR$ axiom and, so, $\inte{\alpha} = 1$ is derived from $\MP$.

As a further illustration of the expressive power of $\PPL$, we want to specify what is envisaged with an oblivious transfer protocol (otp) by expressing the assumed state before and the required state after a run of the protocol. 
To this end, it becomes handy to use the abbreviation
$$\alpha \;\;\text{for}\;\:  \inte{\alpha} =1$$
to which we will return in Section~\ref{sec:cons} for establishing a conservative embedding of $\CPL$ in $\PPL$.

Simplifying from~\cite{rab:81}, an otp is a protocol to be followed by two agents (say John and Mary) so that John sends a bit to Mary, but remains oblivious as to if the bit reached Mary or not, while these two alternatives are equiprobable. Rabin proposed a protocol for solving a more general problem (a message with several bits is to be obliviously sent by John to Mary). The existence of such oblivious transfer protocols is quite significant because by building upon them one can solve other types of cryptographic problems.

The state that is assumed before the run of the protocol and the state required after the run can be specified using the following propositional symbols (each with the indicated intended meaning):
\begin{align*}
\mathsf{JB0} 		&\phantom{{}={}} (\text{John holds bit 0})\displaybreak[1]\\
\mathsf{JB1} 		&\phantom{{}={}} (\text{John holds bit 1})\displaybreak[1]\\
\mathsf{MB0} 		&\phantom{{}={}} (\text{Mary holds bit 0})\displaybreak[1]\\
\mathsf{MB1} 		&\phantom{{}={}} (\text{Mary holds bit 1})\displaybreak[1]\\
\mathsf{JKMB0} 	&\phantom{{}={}} (\text{John knows that Mary holds bit 0})\displaybreak[1]\\
\mathsf{JKMB1} 	&\phantom{{}={}} (\text{John knows that Mary holds bit 1}).
\end{align*}
Indeed, the assumed initial state can be specified by the conjunction of the following $\PPL$ formulas:
\begin{align*}
\mathsf{JB0} \ldisj \mathsf{JB1} 			&\phantom{{}={}} (\text{John holds bit 1 or holds bit 0})\\
(\lneg\mathsf{MB0}) \lconj (\lneg\mathsf{MB1}) 	&\phantom{{}={}} (\text{Mary does not hold bit 1 or bit 0}),
\end{align*}
and the envisaged final state by the conjunction of the following $\PPL$ formulas:
\begin{align*}
\mathsf{JB0} \ldisj \mathsf{JB1} 
				&\phantom{{}={}} (\text{John holds bit 1 or holds bit 0})\displaybreak[1]\\
(\lneg\mathsf{JKMB0}) \lconj (\lneg\mathsf{JKMB1}) 
				&\phantom{{}={}} (\text{John does not know}\\
				&\phantom{{}={}} \text{ what, if anything, Mary holds})\displaybreak[1]\\
\mathsf{MB0} \limp \mathsf{JB0} 
				&\phantom{{}={}} (\text{If Mary holds bit 0 then so does John})\displaybreak[1]\\
\mathsf{MB1} \limp \mathsf{JB1} 
				&\phantom{{}={}} (\text{If Mary holds bit 1 then so does John})\displaybreak[1]\\
\inte{(\mathsf{MB0} \ldisj \mathsf{MB1})}=\frac12 
				&\phantom{{}={}} (\text{Mary holds a bit with probability } \frac12).
\end{align*}
It is also necessary to impose the relevant epistemic requirements:
$$\begin{array}{l}
\mathsf{JKMB0} \limp \mathsf{MB0}\\
\mathsf{JKMB1} \limp \mathsf{MB1}.
\end{array}$$
This example shows the practical interest of developing a probabilistic epistemic dynamic logic as an enrichment of $\PPL$, endeavour that we leave for future work.
 
Observe that in the previous example we only needed a finite number of propositional symbols.
But a key novelty of $\PPL$ is the possibility of working with a denumerable set of propositional symbols.
This capability of $\PPL$ adds a lot to its expressive power.

As an illustration, consider the encoding in $\PPL$ of the halting problem 
(as originally introduced in~\cite{tur:37}): 
$$\text{Does Turing machine $i$ halts on input $j$?}$$
In this case we need the following propositional symbols (with the indicated intended meaning):
$$\begin{array}{lll}
\mathsf{H}_{ijk} & (\text{Machine $i$ halts on input $j$ in $k$ steps}) & \text{for each } i,j,k\in\nats\\
\mathsf{H}_{ij} & (\text{Machine $i$ halts on input $j$}) & \text{for each } i,j\in\nats.
\end{array}$$
Using this denumerable set of propositional symbols, consider the $\PPL$ theory with the following set of proper axioms:
$$\begin{array}{lcl}
\text{Ax}_\mathsf{H} &=&
	\{ \mathsf{H}_{ijk} \limp \mathsf{H}_{ij}: i,j,k\in\nats \} \\[1mm]
	&&\qquad\qquad {\cup} \\[1mm]
	&&\{ \mathsf{H}_{ijk}: \text{machine $i$ halts on input $j$ in $k$ steps} \}.
\end{array}$$
It is straightforward to establish the following fact about this theory for each $i,j\in\nats$:
$$\text{(i) $\text{Ax}_\mathsf{H} \der \mathsf{H}_{ij}$ \;iff\; 
	(ii) machine $i$ halts on input $j$.}$$
Indeed, it is easy to present a derivation for obtaining (i) from (ii).
On the other hand, obtaining (ii) from (i) requires the (strong) soundness of the calculus of $\PPL$, the first result in Section~\ref{sec:soundcomp}.
Observe that $\text{Ax}_\mathsf{H}$ is decidable as required of a set of axioms. 
However, $(\text{Ax}_\mathsf{H})^\der$ is undecidable (since otherwise, thanks to the fact above, the halting problem would also be decidable).
So, this example shows that, in $\PPL$, $\Gamma^\der$ may be undecidable even when $\Gamma$ is assumed to be decidable. But $\emptyset^\der$ is decidable, the last result of Section~\ref{sec:cons}.

The probabilistic capabilities of $\PPL$ would be needed for developing a similar theory for probabilistic Turing machines. To this end, we may take
$$\begin{array}{c} 
	\left\{ \inte{\mathsf{H}_{ijk}} = x_1 \limp \inte{\mathsf{H}_{ij}} \geq x_1: i,j,k\in\nats \right\} \\[1mm]
	{\cup} \\[1mm]
	\left\{ \inte{\mathsf{H}_{ijk}} = p: 
			\text{machine $i$ halts on input $j$ in $k$ steps with probability $p^\reals$} \right\}
\end{array}$$
as the set of proper axioms. 
 
Notwithstanding their simplicity, the examples above should be enough to assess the power of $\PPL$ for describing probabilistic systems and reasoning about them.


\section{Soundness and weak completeness}\label{sec:soundcomp}

In this section we show that the calculus for $\PPL$ is (strongly) sound and weakly complete. Observe that strong completeness is obviously out of question since the $\PPL$ entailment is not compact (as mentioned in Section~\ref{sec:PPL}).
\begin{pth} \em\label{th:sound}
The logic $\PPL$ is sound. \end{pth}
\begin{proof}
The rules are sound. We only check that axiom $\RR$ is sound since the proof of the others is straightforward.
\\
$(\RR)$ is sound. Let $V$ be a stochastic valuation and $\rho$ an assignment over $\reals$. Assume that
$$V \rho \sat \inte{\alpha_j} \mycirc_j p_j \text{ for each } j=1,\dots, k$$
and that the formula
$$\ds\bigforall
		\left(\left(Q^{\alpha_1,\ldots,\alpha_{k+1}}_{\alpha_{11},\ldots,\alpha_{k+1m_{k+1}}} \lconj \bigwedge_{j=1}^k \left(x_{\alpha_{j}} \mycirc_j p_j\right)\right)
					\limp \left(x_{\alpha_{k+1}} \mycirc_{k+1} p_{k+1}\right)\right)$$
					is in $\RCOF$. 
Let $\rho'$ be an assignment over $\reals$ such that,
$$\rho'(x_\alpha)=\Prob_V(\alpha)$$
and $\rho'(x)=\rho(x)$ for every $x \in X_\nats$.
Then,
$$\reals \rho' \satfo Q^{\alpha_1,\ldots,\alpha_{k+1}}_{\alpha_{11},\ldots,\alpha_{k+1m_{k+1}}} \lconj \bigwedge_{j=1}^k \left(x_{\alpha_{j}} \mycirc_j p_j\right).$$
Therefore, 
$$\reals \rho' \satfo x_{\alpha_{k+1}} \mycirc_{k+1} p_{k+1}$$
Hence, $\rho'(x_{\alpha_{k+1}}) \mycirc_{k+1} p_{k+1}^{\reals\rho'}$ and so
$\Prob_V(\alpha_{k+1}) \mycirc_{k+1} p_{k+1}^{\reals\rho}$. Therefore
$V \rho \sat \inte{\alpha_{k+1}} \mycirc_{k+1} p_{k+1}$.
\end{proof}

We now proceed towards the weak completeness of the calculus. We start by proving an important lemma showing that we can move back and forth between 
satisfaction of $\RCOF$ formulas expressing probabilistic reasoning 
and 
satisfaction of $\PPL$ formulas. 

\begin{prop} \em \label{prop:orcftostoval}
Let $\varphi$ be a formula of $\PPL$ and $\alpha_1,\dots,\alpha_k$ be the propositional formulas such that $\inte{\alpha_j} \mycirc_j p_j$
occurs in $\varphi$ for each $j=1,\dots,k$. 
Moreover, let $\Lambda=\{\alpha_{11},\ldots,\alpha_{km_k}\}$ be an adequate set of DNF-conjuncts for $\{\alpha_1,\dots,\alpha_k\}$. Let $\rho$ be an assignment over $\reals$. Assume that
$$\reals \rho \satfo Q^{\alpha_1,\ldots,\alpha_{k}}_{\alpha_{11},\ldots,\alpha_{km_{k}}}.$$
Then, there is a stochastic valuation $V$ such that 
$$V \rho \sat \varphi \quad  \text{iff}\quad  \reals \rho \satfo \psi$$
where $\psi$ is the $\RCOF$ formula obtained from $\varphi$ 
by substituting $x_\alpha \mycirc p$ for each $\PPL$ formula $\inte{\alpha} \mycirc p$.
\end{prop}
\begin{proof}
Let $A \in \fwp B$ and $\fdd_A: \wp A \to [0,1]$ be such that
$$\fdd_A(U)=\frac{1} {2^{|A \setminus B_\Lambda|}} \sum_{\text{\shortstack[c]
{$U' \subseteq  B_\Lambda$\\
$U' \cap A=U \cap B_\Lambda$
}
}} \rho(x_{\phi^{U'}_{B_\Lambda}}).$$
Note that $0\leq  \rho(x_{\phi^{U'}_{B_\Lambda}})\leq 1$ since $\reals \rho \satfo Q^{\alpha_1,\ldots,\alpha_{k}}_{\alpha_{11},\ldots,\alpha_{km_{k}}}$. 
We start by showing that $\fdd_A$ is a finite-dimensional probability distribution. Observe that \\[2mm]
\begin{align*}
\displaystyle \sum_{U \subseteq A} \fdd_A(U)  
            =& \; \displaystyle \sum_{U \subseteq A} \frac{1} {2^{|A \setminus B_\Lambda|}} \sum_{\text{\shortstack[c]
{$U' \subseteq  B_\Lambda$\\
$U' \cap A=U \cap B_\Lambda$
}
}} \rho(x_{\phi^{U'}_{B_\Lambda}})\displaybreak[1] \\[2mm]
          =&\; \displaystyle \frac{1} {2^{|A \setminus B_\Lambda|}} \sum_{U \subseteq A} \sum_{\text{\shortstack[c]
{$U' \subseteq  B_\Lambda$\\
$U' \cap A=U \cap B_\Lambda$
}
}} \rho(x_{\phi^{U'}_{B_\Lambda}})\displaybreak[1] \\[2mm]
=& \; 1
\end{align*}
since
\begin{align*}
\displaystyle \sum_{U \subseteq A} \sum_{\text{\shortstack[c]
{$U' \subseteq  B_\Lambda$\\
$U' \cap A=U \cap B_\Lambda$
}
}} \rho(x_{\phi^{U'}_{B_\Lambda}})
 = & \;
\displaystyle  \sum_{U_1 \subseteq A \setminus B_\Lambda} \sum_{U_2 \subseteq A \cap B_\Lambda} \sum_{\text{\shortstack[c]
{$U' \subseteq  B_\Lambda$\\
$U' \cap A=U_2$
}
}} \rho(x_{\phi^{U'}_{B_\Lambda}}) \displaybreak[1] \\[2mm]
= & \;
\displaystyle  \sum_{U_1 \subseteq A \setminus B_\Lambda} 
\sum_{U' \subseteq B_\Lambda} \rho(x_{\phi^{U'}_{B_\Lambda}}) & (*)\displaybreak[1] \\[2mm]
= & \; 
\displaystyle  \sum_{U_1 \subseteq A \setminus B_\Lambda} 1 & (**) 
\displaybreak[1] \\[2mm]
= & \; 2^{|A \setminus B_\Lambda|}
\end{align*}
where $(*)$ follows from the fact that there is a bijection from
$$\{(U',U_2): U' \subseteq B_\Lambda, U' \cap A = U_2, U_2 \subseteq A \cap B_\Lambda\} \; \text{ to } \; \{U': U'\subseteq B_\Lambda\}$$ and $(**)$ holds 
because $\reals\rho \satfo Q^{\alpha_1,\ldots,\alpha_{k}}_{\alpha_{11},\ldots,\alpha_{km_{k}}}$. \\[2mm]
Now we prove that $\{\fdd_A\}_{A \in \fwp B}$ satisfies the marginal condition. Let $A \subseteq A'$, $A' \subseteq B$ and $U \subseteq A$. Then, \\[2mm]
\begin{align*}
 \fdd_A(U) 
 = & \; \displaystyle \frac{1} {2^{|A \setminus B_\Lambda|}} \sum_{\text{\shortstack[c]
{$U' \subseteq  B_\Lambda$\\
$U' \cap A=U \cap B_\Lambda$
}
}} \rho(x_{\phi^{U'}_{B_\Lambda}})\displaybreak[1] \\[2mm]
= & \; \displaystyle \frac{1} {2^{|A' \setminus B_\Lambda|}} \; \frac{1} {2^{|A \setminus B_\Lambda|}} \; 2^{|A' \setminus B_\Lambda|}\sum_{\text{\shortstack[c]
{$U' \subseteq  B_\Lambda$\\
$U' \cap A=U \cap B_\Lambda$
}
}} \rho(x_{\phi^{U'}_{B_\Lambda}})\displaybreak[1] \\[2mm]
= & \; \displaystyle \frac{1} {2^{|A' \setminus B_\Lambda|}} \; \frac{1} {2^{|A \setminus B_\Lambda|}} \; 2^{|A \setminus B_\Lambda|} \sum_{\text{\shortstack[c]
{$U'' \subseteq  A'$\\
$U'' \cap A=U$
}
}}  \sum_{\text{\shortstack[c]
{$U' \subseteq  B_\Lambda$\\
$U' \cap A'=U'' \cap B_\Lambda$
}
}} \rho(x_{\phi^{U'}_{B_\Lambda}}) & (*)\displaybreak[1] \\[2mm]
= & \; \displaystyle   \sum_{\text{\shortstack[c]
{$U'' \subseteq  A'$\\
$U'' \cap A=U$
}
}}  \frac{1} {2^{|A' \setminus B_\Lambda|}} \sum_{\text{\shortstack[c]
{$U' \subseteq  B_\Lambda$\\
$U' \cap A'=U'' \cap B_\Lambda$
}
}} \rho(x_{\phi^{U'}_{B_\Lambda}})\displaybreak[1] \\[2mm]
= & \; \displaystyle \sum_{\text{\shortstack[c]
{$U'' \subseteq  A'$\\
$U'' \cap A=U$
}
}}\fdd_{A'}(U'')
\end{align*}
where $(*)$ holds since: 
\begin{align*}
\displaystyle 2^{|A' \setminus B_\Lambda|}\sum_{\text{\shortstack[c]
{$U' \subseteq  B_\Lambda$\\
$U' \cap A=U \cap B_\Lambda$
}
}} \rho(x_{\phi^{U'}_{B_\Lambda}}) & \displaybreak[1] \\[2mm]
& \displaystyle \hspace*{-30mm}  = \sum_{U''\subseteq A' \setminus B_\Lambda} \sum_{\text{\shortstack[c]
{$U' \subseteq  B_\Lambda$\\
$U' \cap A=U \cap B_\Lambda$
}
}} \rho(x_{\phi^{U'}_{B_\Lambda}})\displaybreak[1] \\[2mm]
& \displaystyle \hspace*{-30mm} = \sum_{U''' \subseteq A \setminus B_\Lambda} \sum_{\text{\shortstack[c]
{$U'' \subseteq  A'$\\
$U'' \cap A=U$
}
}}  \sum_{\text{\shortstack[c]
{$U' \subseteq  B_\Lambda$\\
$U' \cap A'=U'' \cap B_\Lambda$
}
}} \rho(x_{\phi^{U'}_{B_\Lambda}}) & (**)\displaybreak[1] \\[2mm]
& \displaystyle \hspace*{-30mm} = 2^{|A \setminus B_\Lambda|} \sum_{\text{\shortstack[c]
{$U'' \subseteq  A'$\\
$U'' \cap A=U$
}
}}  \sum_{\text{\shortstack[c]
{$U' \subseteq  B_\Lambda$\\
$U' \cap A'=U'' \cap B_\Lambda$
}
}} \rho(x_{\phi^{U'}_{B_\Lambda}})
\end{align*}
and where $(**)$ holds since 
there is a bijection
$f$ from
$$\{(U'',U'): U'' \subseteq A' \setminus B_\Lambda, U' \subseteq B_\Lambda, U' \cap A = U \cap B_\Lambda\}$$
to
\begin{align*}\{(W''',W'',W'): W''' \subseteq A \setminus B_\Lambda, W'' \subseteq A', & \displaybreak[1] \\[0mm]
  &\hspace*{-40mm} W'' \cap A = U, W' \cap A' = W'' \cap B_\Lambda, W' \subseteq B_\Lambda\}
\end{align*}
such that
$$f(U'',U')= (U'' \cap A, U \cup (U'' \cap (A' \setminus A)) \cup (U' \cap (A' \setminus A)), U').$$
Indeed, \\[1mm]
(a)~$f(U'',U')$ is in the range of $f$: \\[1mm]
(i)~$W''' \subseteq A \setminus B_\Lambda$. Note that $W'''=U'' \cap A$. Hence, $W''' \subseteq A$. Moreover, since $U'' \subseteq B \setminus B_\Lambda$ then 
$W''' \subseteq B \setminus B_\Lambda$.\\[1mm]
(ii)~$W'' \subseteq A'$. Note that $W''= U \cup (U'' \cap (A' \setminus A)) \cup (U' \cap (A' \setminus A))$ and that
$U \subseteq A \subseteq A'$, $U'' \cap (A' \setminus A) \subseteq A'$ and $U' \cap (A' \setminus A) \subseteq A'$. \\[1mm]
(iii)~$W'' \cap A=U$. It is sufficient to note that $W''= U \cup (U'' \cap (A' \setminus A)) \cup (U' \cap (A' \setminus A))$ and that
$U \cap A=U$, $U'' \cap (A' \setminus A) \cap A=\emptyset$ and $U' \cap (A' \setminus A) \cap A=\emptyset$.\\[1mm]
(iv)~$W' \cap A'=W'' \cap B_\Lambda$. It is sufficient to note that $W''= U \cup (U'' \cap (A' \setminus A)) \cup (U' \cap (A' \setminus A))$ and that
$U \cap B_\Lambda=U' \cap A$, $U'' \cap (A' \setminus A) \cap B_\Lambda=\emptyset$ and $U' \cap (A' \setminus A) \cap B_\Lambda = U' \cap (A' \setminus A)$.
So, $W'' \cap B_\Lambda=U' \cap A'=W' \cap A'$. \\[1mm]
(v)~$W' \subseteq B_\Lambda$. Immediate since $U' \subseteq B_\Lambda$.\\[2mm]
(b)~$f$ is injective. Assume that
$f(U''_1,U'_1)=f(U''_2,U'_2)$. Then
$U'_1=U'_2$. Moreover, $U''_1 \cap A= U''_2 \cap A$ and
$$U \cup (U''_1 \cap (A' \setminus A)) \cup (U'_1 \cap (A' \setminus A))=
U \cup (U''_2 \cap (A' \setminus A)) \cup (U'_2 \cap (A' \setminus A)).$$
Observe that $U''_i \cap (A' \setminus A) \cap U= \emptyset$ and $U''_i \cap (A' \setminus A) \cap U'_i \cap (A' \setminus A)= \emptyset$
for $i=1,2$. Hence, $U''_1\cap (A' \setminus A)=U''_2 \cap (A' \setminus A)$. So
\begin{align*}
U''_1 =& \ U''_1 \cap A'\\[1mm]
        =& \ U''_1 \cap (A \cup (A' \setminus A))\displaybreak[1] \\[1mm]
        =& \ (U''_1 \cap A) \cup (U''_1\cap  (A' \setminus A))\displaybreak[1] \\[1mm]
        =& \ (U''_2 \cap A) \cup (U''_2\cap  (A' \setminus A))\displaybreak[1] \\[1mm]
        =& \ U''_2.
\end{align*}
(c)~$f$ is surjective. Let $(W''',W'',W')$ be in the range of $f$. Take
$$U''=W''' \cup ((W'' \setminus A ) \setminus B_\Lambda), \; U'= W'.$$
(i)~$(U'',U')$ is in the domain of $f$: \\[1mm]
- $U'' \subseteq A' \setminus B_\Lambda$. Note that $U'' = W''' \cup ((W'' \setminus A ) \setminus B_\Lambda)$,  $W''' \subseteq A \subseteq A'$ and
$W''' \subseteq B \setminus B_\Lambda$. So, $W''' \subseteq A' \setminus B_\Lambda$. On the other hand,
$(W'' \setminus A ) \setminus B_\Lambda \subseteq A'$ since $W'' \subseteq A'$ and  $(W'' \setminus A ) \setminus B_\Lambda \subseteq B \setminus B_\Lambda$.
So, $U'' \subseteq A' \setminus B_\Lambda$.\\[1mm]
- $U' \subseteq B_\Lambda$. Immediate since $W' \subseteq B_\Lambda$.\\[1mm]
- $U' \cap A=U \cap B_\Lambda$. Observe that 
\begin{align*}
U \cap B_\Lambda =& \ W'' \cap B_\Lambda \cap A)\displaybreak[1] \\[1mm]
              =& \ W' \cap A' \cap A)\displaybreak[1] \\[1mm]
        =& \ W' \cap A)\displaybreak[1] \\[1mm]
                =& \ U' \cap A.
\end{align*}

(ii)~$f(U'',U')=(W''',W'',W')$. Indeed:\\[1mm]
- $U''\cap A=W'''$. In fact
\begin{align*}
U'' \cap A =& \ (W''' \cup ((W'' \setminus A ) \setminus B_\Lambda)) \cap A)\displaybreak[1] \\[1mm]
                =& \  (W''' \cap A) \cup (((W'' \setminus A ) \setminus B_\Lambda) \cap A))\displaybreak[1] \\[1mm]
                =& \ W''' \cap A )\displaybreak[1] \\[1mm]
                =& \ W'''.
\end{align*}
- $U \cup (U'' \cap (A' \setminus A)) \cup (U' \cap (A' \setminus A))=W''$. In fact
\begin{align*}
U \cup (U'' \cap (A' \setminus A)) \cup (U' \cap (A' \setminus A)) = \displaybreak[1] \\[1mm]
               & \hspace*{-48mm}  U \cup ((W''' \cup ((W'' \setminus A ) \setminus B_\Lambda)) \cap (A' \setminus A)) \cup (W' \cap (A' \setminus A))=\displaybreak[1] \\[1mm]
                & \hspace*{-48mm}  U \cup (((W'' \setminus A ) \setminus B_\Lambda) \cap (A' \setminus A)) \cup (W' \cap (A' \setminus A))=\displaybreak[1] \\[1mm]
                & \hspace*{-48mm}  U \cup ((W'' \setminus A ) \setminus B_\Lambda)  \cup ((W'' \cap B_\Lambda) \setminus A)=\displaybreak[1] \\[1mm]
                & \hspace*{-48mm}  U \cup (W'' \setminus A)  =\displaybreak[1] \\[1mm]
               & \hspace*{-48mm}  (W'' \cap A) \cup (W'' \setminus A) = \displaybreak[1] \\[1mm]
                & \hspace*{-48mm}  W''.
\end{align*}
Hence, using Kolmogorov's existence theorem, there exists a unique stochastic valuation $V$ having these finite-dimensional distributions.\\[2mm]
Finally, we show, by induction on the structure of $\varphi$, that
$$V\rho \sat \varphi \quad \text{iff} \quad \reals \rho \satfo \psi.$$ 
Base: $\varphi$ is $\inte{\alpha_j} \mycirc_j p_j$. Observe first that:
\begin{align*}
\Prob_V(\alpha_j) =& \; \ds \sum_{U \in \termvalue{\alpha_j}} \Prob(V_{B_{\alpha_j}} = U)\displaybreak[1] \\[2mm]
                =& \; \ds \sum_{U\in\{v\cap B_{\Lambda}, v \satc \alpha_j\}} \Prob(V_{B_{\Lambda}} = U) & (*)\displaybreak[1] \\[2mm]
                =& \; \ds \sum_{v\cap B_{\Lambda}, v \satc \alpha_j} \fdd_{B_{\Lambda}}(v \cap B_{\Lambda})\displaybreak[1] \\[2mm]
                =& \;  \ds \sum_{v\cap B_{\Lambda}, v \satc \alpha_j} \rho(x_{\phi^{v \cap B_{\Lambda}}_{B_{\Lambda}}})\displaybreak[1] \\[2mm]
                =& \; \ds \sum_{\ell =1,\dots, m_j} \rho(x_{\alpha_{j\ell}}) & (**)\displaybreak[1] \\[2mm]
                =& \; \ds \rho(x_{\alpha_j}) & ({*}{*}{*})
\end{align*}
where $(*)$ holds by Proposition~\ref{prop:mardisj}, $(**)$ holds since $\{\alpha_{11},\ldots,\alpha_{km_k}\}$ is an adequate set of DNF-conjuncts for $\{\alpha_1,\dots,\alpha_k\}$, and $({*}{*}{*})$ holds since
$\reals\rho \satfo Q^{\alpha_1,\ldots,\alpha_{k}}_{\alpha_{11},\ldots,\alpha_{km_{k}}}$.
 Then,  \\[1mm]
($\from$)~Assume that $\reals\rho \satfo x_{\alpha_j} \mycirc_j p_j$. Then
$\rho(x_{\alpha_j}) \mycirc_j  p_j^{\reals \rho}$. So, $\Prob_V(\alpha_j) \mycirc_j  p_j^{\reals \rho}$.
Hence, $V \rho \sat \varphi$. \\[1mm]
$(\to)$ Assume that $V \rho \sat \inte{\alpha_j} \mycirc_j p_j$. Then
$\Prob_V(\alpha_j) \mycirc_j  p_j^{\reals \rho}$. Thus, $\rho(x_{\alpha_j}) \mycirc_j  p_j^{\reals \rho}$
and so $\reals \rho \satfo \psi$. \\[2mm]
Step: $\varphi$ is $\varphi_1 \limp \varphi_2$. Then
$$
\begin{array}{cl}
\reals \rho \satfo \psi_1 \limp \psi_2  \\[1mm]
      \text{iff} & \\[1mm]
     \reals \rho \not \satfo \psi_1 \text{ or } \reals \rho\satfo \psi_2\\[1mm]
           \text{iff} & \qquad\text{ IH } \\[1mm]
         V\rho \not \sat \varphi_1 \text{ or } V\rho \sat \varphi_2\\[1mm]
      \text{iff} & \\[1mm]
      V\rho \sat \varphi
\end{array}
$$
where $\psi_1$ and $\psi_2$ are formulas obtained from $\varphi_1$ and $\varphi_2$, respectively, by replacing each formula $\inte{\alpha} \mycirc p$ by $x_\alpha \mycirc p$.
\end{proof}

\begin{prop} \em \label{prop:dnf}
Let $\varphi \in L_\PPL$. Then, there is $\psi \in L_\PPL$ such that 
$$\der \varphi \displaystyle \leqv \psi$$
and $\psi$ is in disjunctive normal form. 
Moreover, if $\varphi$ is consistent then there is a conjunction of literals in $\psi$ that is also consistent.
\end{prop}

\begin{pth} \em
The logic $\PPL$ is weakly complete.
\end{pth}
\begin{proof}\ \\
Let $\varphi \in L_\PPL$. 
Assume that $\not \der \varphi$. We proceed to show that $\not \ent \varphi$.
First observe that $\lneg\varphi$ must be consistent in the sense that $\lneg\varphi \not\der \lfalsum$
because otherwise from $\lneg\varphi$ one would be able to derive every formula, including in particular, 
$\lneg\varphi \limp \varphi$, and in that case one would have $\der\varphi$. 
By Proposition~\ref{prop:dnf}, 
$$\der (\lneg \varphi) \displaystyle \leqv \bigvee_{m \in M} \eta_m$$ where each disjunct is a conjunction of literals. Since $\lneg\varphi$ is consistent, at least one of the disjuncts must also be consistent.
Let $\eta_m$ be one such consistent disjunct.
In order to show that $\not \ent \varphi$ it is enough to show that $\lneg\varphi$ is satisfiable. 
Hence, it is enough to show that there is one satisfiable disjunct. Indeed, $\eta_m$ is satisfiable.
Towards a contradiction, assume that there are no $V$ and $\rho$ such that 
$V \rho \sat \eta_m$ holds.
Let $\eta_m$ be of the form 
$$(\inte{\alpha_1} \mycirc_1 p_1) \lconj\dots \lconj (\inte{\alpha_k} \mycirc_k p_k).$$ 
Let 
$$\{\alpha_{11}, \ldots, \alpha_{1m_1}, \ldots, \alpha_{k1}, \ldots, \alpha_{km_k}\} \subset L$$
be an adequate set of DNF-conjuncts for $\{\alpha_1,\ldots,\alpha_k\} \subset L$.
Then, by Proposition~\ref{prop:orcftostoval}, there would not exist $\rho$ and such that
$$\reals \rho \satfo Q^{\alpha_1,\ldots,\alpha_{k}}_{\alpha_{11},\ldots,\alpha_{km_{k}}}
					\lconj \left(\bigwedge_{j=1}^{k} \left(x_{\alpha_j} \mycirc_j p_j\right)\right).$$
Hence, we would have 
$$\ds\bigforall
		\left(\left(Q^{\alpha_1,\ldots,\alpha_{k}}_{\alpha_{11},\ldots,\alpha_{km_{k}}} \lconj \bigwedge_{j=1}^k \left(x_{\alpha_{j}} \mycirc_j p_j\right)\right)
					\limp \lfalsum \right)\in \\[1mm] \hspace*\fill \RCOF.$$
Then, by $\RR$, we would establish $$\eta_m \der \lfalsum$$ 
in contradiction with the consistency of $\eta_m$.
\end{proof}

\section{Conservativeness and decidability}\label{sec:cons}

In this section we start by working towards showing that $\PPL$ is a conservative extension of classical propositional logic modulo a translation of classical formulas into the language of $\PPL$. 

Given $\alpha \in L$, we denote by $\alpha^*$ the $\PPL$ formula $\inte{\alpha} =1$. Moreover, given $\Delta  \subseteq L$, we denote by $\Delta^*$ the set $\{\delta^*: \delta \in \Delta\}$.

\begin{prop} \em
Letting $\Delta \cup \{\alpha\} \subseteq L$,
If $\Delta \der_c \alpha$ then $\Delta^* \der \alpha^*$.
\end{prop}
\begin{proof}\ \\
Just observe that 
if $\alpha_1,\dots,\alpha_n$ is a derivation sequence of $\alpha=\alpha_n$ from $\Delta$ in $\CPL$,
then, making good use of $\TAUT^*$ and $\MP^*$ 
(the admissible rules established in Subsection~\ref{subsec:pplexamples}),
$\alpha_1^*,\dots,\alpha_n^*$ is a derivation sequence of $\alpha^*$ from $\Delta^*$ 
in $\PPL$.
\end{proof}

\begin{pth} \em
Let $\Delta \cup \{\alpha\} \subseteq L$. Then
$$\Delta^*\ent \alpha^* \quad \text{ iff } \quad \Delta \entc \alpha.$$ 
\end{pth}
\begin{proof}\ \\
$(\to)$ Assume that $\Delta^*\ent \alpha^*$.
Let $v$ be a (classical) valuation such that $v \satc \delta$ for every $\delta \in \Delta$.
Then, by Proposition~\ref{prop:satvssatp}, $\Prob_{V^v}(\delta)\geq 1$ for every $\delta \in \Delta$. Hence, 
$V^v \sat \inte{\delta} =1$ for every $\delta \in \Delta$. Thus, $V^v \sat \inte{\alpha} =1$ and, so,
$\Prob_{V^v}(\alpha)= 1$. Therefore,
using the same proposition, $v \satc \alpha$.\\[2mm]
$(\from)$ Assume that $\Delta \entc \alpha$. 
Then, thanks to the previous proposition, $\Delta^* \der \alpha^*$ and, so, by Theorem~\ref{th:sound},
$\Delta^*\ent \alpha^*$. 
\end{proof}

We now concentrate on the decidability of the $\PPL$ validity problem. For this purpose we assume given the following two algorithms. 
Let $\cA_{\DNF}$ be an algorithm that receives a propositional formula $\alpha$ and a set of propositional symbols $A\supseteq B_\alpha$,
and returns a set $\{\beta_1, \dots,\beta_m\}$ of conjunctions of literals  
such that each $B_{\beta_i} = A$, 
$\beta_1 \ldisj \dots\ldisj \beta_m$ is a disjunctive normal form of $\alpha$ 
and $\not \entc \beta_i\leqv \beta_j$ for $1 \leq i \neq j \leq m$. 
Furthermore, let $\cA_{\RCOF}$ be an algorithm for deciding the validity of sentences in $\RCOF$. 

The procedure in Figure~\ref{alg:ppl} receives a $\PPL$ formula and returns true whenever the formula is valid and false otherwise. Indeed, the following theorem establishes that the execution of $\cA_{\PPL}$ always terminates and does so with the correct output.

\begin{figure}
	       \hrule  \vspace*{4mm}
	       Input: formula $\varphi$ of $\PPL$.
\begin{enumerate}
  \item Let $B_\varphi := \{B_j: B_j \in B \text{ and } B_j \text{ occurs in } \varphi\}$;
  \item Let $\{\alpha_1,\dots,\alpha_k\} :=\{\alpha: \inte{\alpha} \mycirc p \text{ occurs in } \varphi \}$;
  \item Let $\psi$ be the formula obtained from $\varphi$ by replacing each formula $\inte{\alpha} \mycirc p$ by $x_\alpha \mycirc p$;
  \item For each $j =1,\dots, k$:
        \begin{enumerate}
             \item Let $\{\alpha_{j1},\dots,\alpha_{jm_j}\} := \cA_{\DNF}(\alpha_j, B_\varphi)$;
         \end{enumerate}
   \item Return $\cA_\RCOF\left(\bigforall \left(Q^{\alpha_1,\ldots,\alpha_{k}}_{\alpha_{11},\ldots,\alpha_{km_{k}}} \limp \psi\right)\right)$.
  \end{enumerate}
     \hrule  \vspace*{4mm}
\caption{Algorithm $\cA_{\PPL}$.}
\label{alg:ppl}
\end{figure}

\begin{pth} \em\label{th:soundalg}
The procedure $\cA_{\PPL}$ is an algorithm. Moreover, $\cA_{\PPL}$ is correct. 
\end{pth}
\begin{proof}\ \\
It is straightforward to verify that the execution of $\cA_{\PPL}$ always terminate, returning either true or false, so we focus on correctness:\\[2mm]
(i)~We start by showing that if $\cA_\PPL(\varphi)$ is true then $\varphi$ is a valid formula of $\PPL$. 
Let $\varphi$ be a formula of $\PPL$. Assume that $\cA_\PPL(\varphi)$ is true. Then, 
$$\cA_\RCOF\left(\bigforall \left(Q^{\alpha_1,\ldots,\alpha_{k}}_{\alpha_{11},\ldots,\alpha_{km_{k}}} \limp \psi\right)\right)$$
is true. Let $V$ be a stochastic valuation and $\rho$ an assignment. Let $\rho'$ be an assignment over $\reals$ such that, 
$$\rho'(x_\alpha)=\Prob_V(\alpha)$$
and $\rho'(x)=\rho(x)$ for every $x \in X_\nats$. We now show that
$$\reals \rho' \satfo Q^{\alpha_1,\ldots,\alpha_{k}}_{\alpha_{11},\ldots,\alpha_{km_{k}}}.$$
Recall that $\Prob_V$ is an Adams' probability assignment (Theorem~\ref{th:probassV}).
Hence, $\Prob_V$ satisfies Adams' postulates. Therefore:\\[1mm]
$$\reals \rho' \satfo \bigwedge_{U\subseteq B_{\varphi}} 0 \leq x_{\phi_{B_{\varphi}}^U} \leq 1$$
since $\rho'(x_{\phi_{B_{\varphi}}^U})=\Prob_V(\phi_{B_{\varphi}}^U)$ and using postulate P1.
Moreover, 
$$\reals \rho' \satfo \sum_{U\subseteq B_{\varphi}} x_{\phi_{B_{\varphi}}^U}=1$$
by postulates P2 and P4
since 
$$\entc  \bigvee_{U \subseteq B_{\varphi}} \phi_{B_{\varphi}}^U$$ by Proposition~\ref{prop:disjtottaut}, and 
using the fact that 
$\rho'(x_{\phi_{B_{\varphi}}^U})=\Prob_V(\phi_{B_{\varphi}}^U)$ and
$\entc \lneg(\alpha_{j\ell} \lconj \alpha_{j\ell'})$ for every $j=1,\dots,k$ and $1 \leq \ell \neq \ell' \leq m_j$. 
Finally, 
$$\bigwedge_{j=1}^{k'}\left(x_{\alpha_j}=\sum_{\ell=1}^{m_j} x_{\alpha_{j\ell}}\right).$$
since 
$$\Prob_V(\alpha_j)=\Prob_V(\bigvee_{1\leq\ell \leq m_j} \alpha_{j\ell})$$
by postulate P3, and since, by postulate P4 
$$\Prob_V(\bigvee_{1\leq\ell \leq m_j} \alpha_{j\ell})=\sum_{\ell=1}^{m_j} \Prob_V({\alpha_{j\ell}}).$$
Hence, $\reals \rho' \satfo Q^{\alpha_1,\ldots,\alpha_{k}}_{\alpha_{11},\ldots,\alpha_{km_{k}}}$. Moreover
$$\reals \satfo \bigforall \left(Q^{\alpha_1,\ldots,\alpha_{k}}_{\alpha_{11},\ldots,\alpha_{km_{k}}} \limp \psi\right)$$
and so $\reals \rho' \satfo \psi$. \\[2mm]
We now show, by induction on the structure of $\varphi$,  that
$$\reals \rho' \satfo \psi \quad  \text{ iff } \quad  V \rho \sat \varphi.$$
Base: $\varphi$ is $\inte{\alpha} \mycirc p$. Then
$$\reals \rho' \satfo x_\alpha \mycirc p \text{ iff } \rho'(x_\alpha) \mycirc p^{\reals\rho'} 
\text{ iff } \Prob_V(\alpha) \mycirc p^{\reals\rho} \text{ iff }  V\rho \sat \inte{\alpha} \mycirc p.$$
Step: $\varphi$ is $\varphi_1 \limp \varphi_2$. Then
$$
\begin{array}{cl}
\reals \rho' \satfo \psi_1 \limp \psi_2  \\[1mm]
      \text{iff} & \\[1mm]
     \reals \rho' \not \satfo \psi_1 \text{ or } \reals \rho' \satfo \psi_2\\[1mm]
           \text{iff} & \text{ IH } \\[1mm]
         V\rho \not \sat \varphi_1 \text{ or } V\rho \sat \varphi_2\\[1mm]
      \text{iff} & \\[1mm]
      V\rho \sat \varphi
\end{array}
$$
where $\psi_1$ and $\psi_2$ are formulas obtained from $\varphi_1$ and $\varphi_2$, respectively, by replacing each formula $\inte{\alpha} \mycirc p$ by $x_\alpha \mycirc p$. Therefore, $\ent \varphi$.\\[2mm]
(ii)~We now show that if $\cA_\PPL(\varphi)$ is false  then $\varphi$ is not a valid formula of $\PPL$.  By contraposition, assume that $\ent \varphi$. We prove that, for every assignment $\rho$ over $\reals$,
$$\reals\rho \satfo 
	\bigforall \left(Q^{\alpha_1,\ldots,\alpha_{k}}_{\alpha_{11},\ldots,\alpha_{km_{k}}} \limp \psi\right).$$
Assume that
$$\reals \rho \satfo Q^{\alpha_1,\ldots,\alpha_{k}}_{\alpha_{11},\ldots,\alpha_{km_{k}}}.$$
Let $V$ be the stochastic valuation induced by $\rho$ as defined  in Proposition~\ref{prop:orcftostoval}. 
Then, $V \rho \sat \varphi$ since $\ent \varphi$ and, so, by the same proposition,
$\reals \rho \satfo \psi$. \\[1mm]
Therefore,
$$\cA_\RCOF\left(\bigforall \left(Q^{\alpha_1,\ldots,\alpha_{k}}_{\alpha_{11},\ldots,\alpha_{km_{k}}} \limp \psi\right)\right)$$
returns true and so does $\cA_\PPL(\varphi)$. 
\end{proof}

\section{Concluding remarks}\label{sec:concs}

Always within the setting of  propositional logic we looked at ways of introducing probabilistic reasoning into logic.

First, towards assigning probabilities to valuations we proposed to look at a random valuation as a stochastic process indexed by the set of propositional symbols. This novel notion (of stochastic valuation as we called it) has the advantage of allowing the use of Kolmogorov's existence theorem for moving from the finite-dimensional probability distributions to the distribution in the underlying probability space. In particular, the existence theorem was quite useful in establishing the equivalence between Adams' probability assignments to formulas and stochastic valuations. 

Afterwards, we investigated the notion of probabilistic entailment in the scenario of leaving the propositional language unchanged. We found that the notion previously proposed in the literature does not enjoy key properties of logical consequence. For this reason, we proposed a more appropriate notion of probabilistic entailment. This notion turned out to be identical to classical entailment.

Since nothing is gained by introducing probabilities without changing the language,
we decided to set-up a small enrichment ($\PPL$) of classical propositional logic
by adding a language construct, 
inspired by~\cite{hal:90,hei:01,pmat:acs:css:05a}, 
that allows the constraining (without nesting) of the probability of a formula.

The resulting extension of classical propositional logic was shown to be rich enough for setting-up interesting theories and easy to axiomatize by relying on the decidable theory of real closed ordered fields 
($\RCOF$). 
In due course, we proved that the extension is conservative and still decidable.

Note that a finite set of propositional symbols was assumed in the logics previously reported in the literature that allow the constraining of the probability of a formula, while $\PPL$ provides the means for working with a denumerable set of propositional symbols, adding a lot to its expressive power. To this end, the notion of stochastic valuation as a stochastic process was a key ingredient.

Concerning future work, it seems worthwhile to investigate other meta-properties of $\PPL$, starting with bounding the complexity of its decision problem. 
We expect this complexity to be much lower than the complexity of $\RCOF$ theoremhood, since we only need to recognize $\RCOF$ theorems of a very simple clausal form.
 
Strong completeness of the $\PPL$ axiomatization was out of question because the semantics over $\reals$ led to a non-compact entailment.
Relaxing the semantics by allowing any model of $\RCOF$ may open the door to establishing strong completeness. 
Clearly, one should start by investigating whether Kolmogorov existence theorem can be carried over to every $\RCOF$ model.

The relevance of abduction in probabilistic reasoning was recognized in~\cite{hae:11}.
One would like to be able to compute the required probability of the conjunction of the relevant hypotheses 
in order to ensure an envisaged probability for the conclusion.
In this front, we expect to be able to find inspiration in the calculus presented in~\cite{acs:jfr:css:pmat:13},
given its abductive nature, towards developing an abduction calculus for $\PPL$.

\section*{Acknowledgments}

The authors are grateful to Juliana Bueno-Soler and Walter Carnielli for reawakening their interest on the probabilization of propositional logic.
This work was supported by Fundação para a Ciência e a Tecnologia by way of grant UID/MAT/04561/2013 to 
Centro de Matemática, Aplicações Fundamentais e Investigação Operacional of Universidade de Lisboa (CMAF-CIO)
and by the European Union's Seventh Framework Programme for Research (FP7) 
namely through project LANDAUER (GA 318287).




\end{document}

